\documentclass[reqno,english]{amsart}

\usepackage{macros}
\usepackage{algorithm}
\usepackage{algpseudocode}

\newcommand{\stream}{\mathcal{S}}
\newcommand{\alg}{\mathcal{A}}
\newcommand{\tree}{\mathcal{T}}
\newcommand{\ltree}{\mathcal{T}_l}
\newcommand{\rtree}{\mathcal{T}_r}
\newcommand{\typ}{\mathrm{typ}}

\newcommand{\pbeta}{\hat{\beta}}

\newcommand{\beq}[1]{\begin{equation}\label{#1}}
\newcommand{\enq}[0]{\end{equation}}

\usepackage[numbers, sort&compress]{natbib}

\title[Online Permutation Embedding]{Online Permutation Embedding: \\ Optimal Stopping and Scaling Laws}
\author{Dylan J. Altschuler, Quentin Dubroff, and Konstantin Tikhomirov}
\date{\today}

\begin{document}

\begin{abstract}
We study optimal online algorithms for embedding a permutation $\pi$ of $[k]$ into an iid stream of uniform $[0,1]$ random variables. This problem is a broad generalization of the classical online monotone subsequence selection problem, recovered in the special case $\pi=\mathrm{Id}_k$. Our first contribution is an efficiently solvable dynamic program for the optimal embedding time of any $k$-permutation $\pi$. This dynamic program also yields an explicit optimal online embedding algorithm. We then investigate the asymptotic scaling of the optimal embedding time for uniformly random target permutations, as well as the extremal problem of identifying the permutations with largest expected online embedding time. Our second main result shows that, to first order, random permutations are strictly faster to embed than monotone permutations, which in turn are strictly faster to embed than the extremal permutations. This separation stands in sharp contrast to prevailing conjectures and heuristics in the offline theory of permutation embeddings.

\end{abstract}

\maketitle

\setcounter{tocdepth}{1}
\tableofcontents

\newpage

\section{Introduction}
Let \(\Pi_k\) denote the set of permutations of \([k] := \{1,2,\dots,k\}\),
let \(\pi\in\Pi_k\), and let
$\stream=(X_1,X_2,\ldots)$ be a stream of independent uniform $[0,1]$ random variables. We say that
$(X_{\tau_1},\ldots,X_{\tau_k})$ is an order-isomorphic embedding of $\pi$ into
the stream $\stream$ if, for all $i,j\in[k]$,
\[
    \pi(i)<\pi(j)
    \qquad\Longleftrightarrow\qquad
    X_{\tau_i}<X_{\tau_j}.
\]
An online algorithm for embedding $\pi$ is a strictly increasing sequence of stopping times adapted to the natural filtration of $\stream$ such that
$(X_{\tau_1},\ldots,X_{\tau_k})$ is almost surely order-isomorphic to $\pi$.
That is, the algorithm must decide immediately and irrevocably whether to include each of the sequentially revealed samples into the embedding.

The purpose of this work is to
initiate a systematic study of the optimal online embedding time for
permutations. Letting $\mathrm{OA}(\pi)$ denote the set of online embedding
algorithms for $\pi$, we define
\begin{equation}\label{eq:intro-beta}
    \beta(\pi) := \inf_{(\tau_1,\dots,\tau_k) \in \mathrm{OA}(\pi)} \E{\tau_k}\,.
\end{equation}
The study of $\beta$ offers a broad extension of the online longest increasing subsequence (LIS) problem, a classical model at the intersection of combinatorics, decision theory, and optimal stopping \cite{steele} that corresponds to the special case $\pi=\mathrm{Id}_k$. There, an embedding is precisely an increasing subsequence of length $k$.
The size-focused online LIS problem was initiated by Samuels and Steele \cite{steele}. Its time-focused dual was studied by Arlotto, Mossel, and Steele \cite{Mossel2015}, who showed
\[
    \beta(\mathrm{Id}_k)=\frac{k^2}{2}+o(k^2).
\]
That is, an increasing subsequence of length $k$ can be selected from a random stream of length $k^2/2$ in an online fashion. In contrast, classical results of Logan and Shepp \cite{logan-shepp} and Vershik and Kerov \cite{vershik-kerov} on Ulam's problem \cite{ulam1961} imply that such a stream actually contains an increasing subsequence of length \((\sqrt{2}+o(1))k\). Thus, for embedding monotone permutations, the loss from making irrevocable online decisions is only a constant factor. This online/offline comparison places online LIS in the same broad paradigm as prophet inequalities, where one compares sequential decision rules with a hindsight benchmark (e.g., famously, the Secretary Problem). The general permutation embedding problem introduced in the present article is motivated in large part by the question of how the online/offline gap varies with the structure of $\pi$.

\subsection{Bellman Equation} Over the past four decades, a substantial literature has emerged around online LIS, including Poissonized formulations, connections to sequential knapsack and bin-packing, and random-permutation variants
\cite{steele,bruss-robertson,gnedin1999, Gnedin2000Permutation, bruss-delbaen2004,arlotto-nguyen-steele,GnedinSeksenbayev2021,permutation,Bruss21}. Additionally, several works investigate embedding unimodal, finite-block
monotone, and alternating subsequences \cite{arlotto-steele-unimodal,ArlottoChenSheppSteele2011,ArlottoSteele2014}.
However, much of this theory exploits a highly specific self-similar
structure common to these models. For example, when embedding $\mathrm{Id}_k$, after the first element is embedded at value \(x\), the remaining task is a fresh instance of the problem of embedding $\mathrm{Id}_{k-1}$, rescaled to \([x,1]\). This scaled regeneration enables tractable recursive analysis \cite{Mossel2015,steele}, and is present in the other mentioned previously analyzed models.

The recursive structure of general permutations is substantially more complicated. After embedding the first element of a target permutation, the remaining elements separate into two distinct subpatterns, one which must be embedded below the first element and another which must be embedded above. However, these subpatterns are interlaced in general, creating substantial technical obstacles to direct analysis. 
Nonetheless, our first main contribution is a Bellman-type recursive formula for $\beta(\pi)$, holding for arbitrary $\pi$. Let $L(\pi) \in \Pi_{\pi(1)-1}$ and $R(\pi) \in \Pi_{k-\pi(1)}$ denote the (unique) permutations which are respectively order-isomorphic to 
\[
    (\pi(j): j \in [k], \pi(j)<\pi(1)) \quad \text{ and } \quad (\pi(j):j \in [k],\pi(j)>\pi(1)).
\]
We will refer to $L(\pi)$ and $R(\pi)$ as the left and right (or lower and upper) subpermutations of $\pi$. 

\begin{theorem}[Expected embedding time of fixed patterns]
Set $\beta(\emptyset) = 0$. Then,
            \begin{equation}\label{eq:rec fixed}
        \beta(\pi)
        =
        \min_{0 \le a < b \le 1}
        \cb{
            \frac{1}{b-a}
            +
            \frac{1}{b-a}
            \int_{a}^{b}
                \frac{\beta(L(\pi))}{x}
                +
                \frac{\beta(R(\pi))}{1-x}
            \, dx
        }.
    \end{equation}
    Here, $\beta(\emptyset)/0 := 0$ by convention. That is, terms from empty subpermutations are omitted from the integrand. 
\end{theorem}

Thus $\beta(\pi)$, although defined through an infinite-dimensional optimization
over stopping times, admits a finite-dimensional dynamic program. This
dynamic program gives an efficient method for computing $\beta(\pi)$ and yields an explicit optimal embedding strategy.
In this strategy, the first element of $\pi$ is embedded by the hitting time
\[
    \tau_1:=\min\{t:X_t\in[a_*,b_*]\},
\]
where $(a_*,b_*)$ is an optimizer in \eqref{eq:rec fixed}. Then, conditional on
$X_{\tau_1}=x$, the algorithm recursively embeds $L(\pi)$ in the lower thinned
stream $[0,x]$ and $R(\pi)$ in the upper thinned stream $[x,1]$. This procedure is formalized as Algorithm~\ref{alg:canonical-general2} below, and it admits simple and efficient practical implementation.

\begin{figure}[!t]
\begin{center}
    \includegraphics[width=0.8\linewidth]{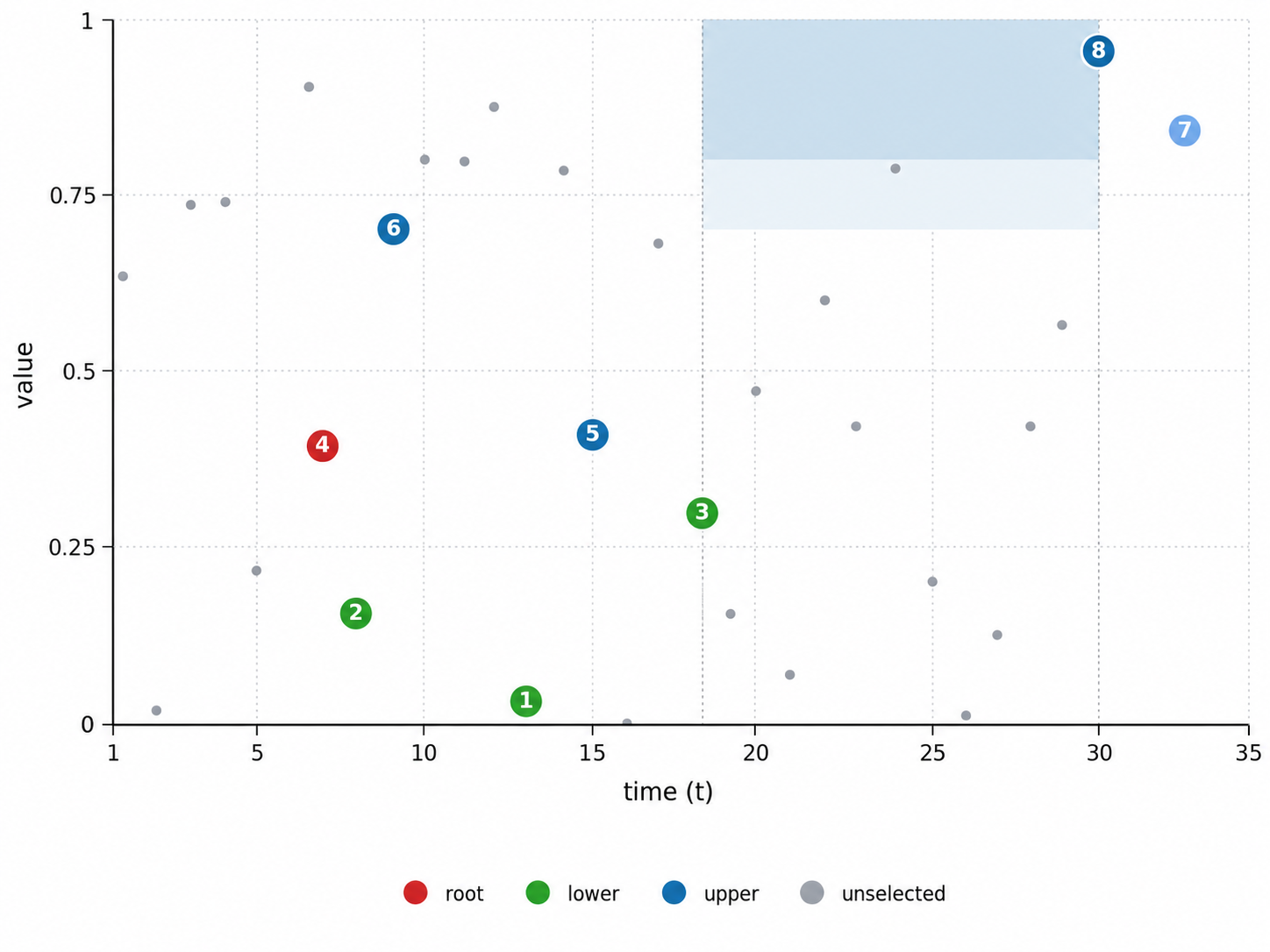}
\end{center}
\captionof{figure}{
    Execution of the optimal online algorithm for embedding
    $\pi=(4,2,6,1,5,3,8,7)$ into a random stream. The light blue interval is the admissible
    region for embedding 8. Anticipating the space needed to embed 7, the optimal rule for embedding 8 accepts the first stream arrival in the dark blue interval. 
}
\label{fig:perm-embed}\end{figure}

\begin{remark}\label{rem:tree-shape-invariance}
An immediate corollary of \eqref{eq:rec fixed} is that $\beta(\pi)$ only depends on the isomorphism class of $\tree_\pi$, i.e., the \textit{unlabeled} tree decomposition of $\pi$. In particular, for every \(\pi\in\Pi_k\), there exists a permutation $\pi_{231}$ which is 231-avoiding (i.e., does not contain any subsequence that is order-isomorphic to $231$) with
\[\beta(\pi)=\beta(\pi_{231}).\] 
Indeed, if $\pi'$ is the permutation obtained from $\pi$ by concatenating $\pi(1)$, $L(\pi)$ and $R(\pi)$ (where $R(\pi)$ is shifted up by $\pi(1)$), then \eqref{eq:rec fixed} implies that $\beta(\pi') = \beta(\pi)$. Iterating this procedure on $L(\pi)$ and $R(\pi)$ yields a 231-avoiding permutation. 
A notable consequence of this observation is the apparent dramatic reduction in complexity for the set of inputs to $\beta$. While there are $k!$ permutations, the number of 231-avoiding permutations is given by the $k$-th Catalan number, which grows only exponentially in $k$.
\end{remark}

We next show that $\beta$ is the \emph{probabilistic} threshold for embedding, in addition to being the expected threshold. In particular, we show that the run-time of the optimal algorithm concentrates about its expectation $\beta(\pi)$, and conversely that there is no embedding algorithm providing a high-probability guarantee of embedding faster than $\beta(\pi)$. For a given algorithm $\alg \in OA(\pi)$, let $T_\alg$ be its stopping time, the time at which the final value of $\pi$ is embedded.

\begin{figure}
    \centering
    \begin{minipage}[t]{0.485\linewidth}
        \vspace{0pt}
        \centering
        \includegraphics[width=\linewidth]{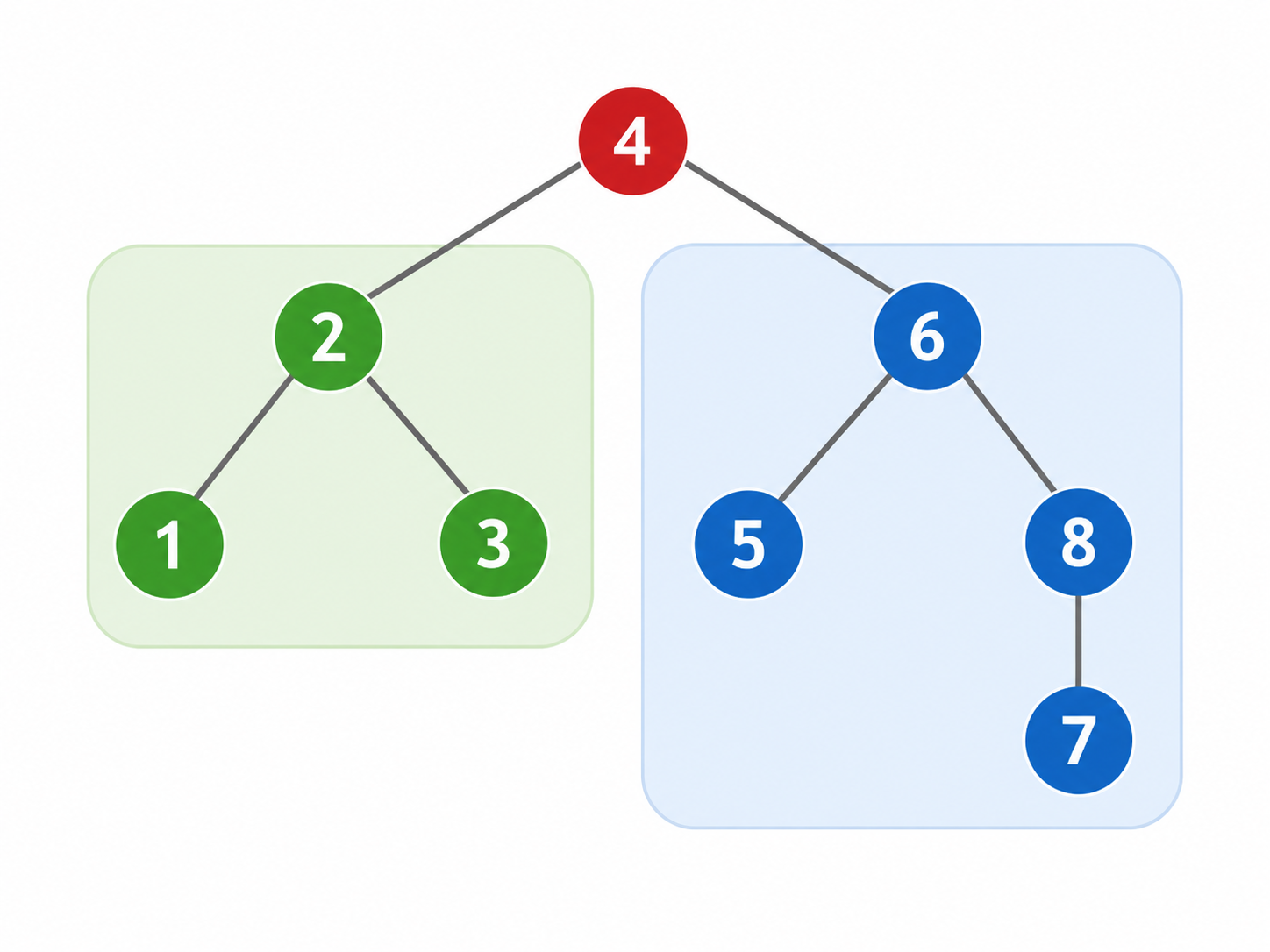}
    \end{minipage}
    \hfill
    \begin{minipage}[t]{0.485\linewidth}
        \vspace{0pt}
        \centering
        \includegraphics[width=.9\linewidth]{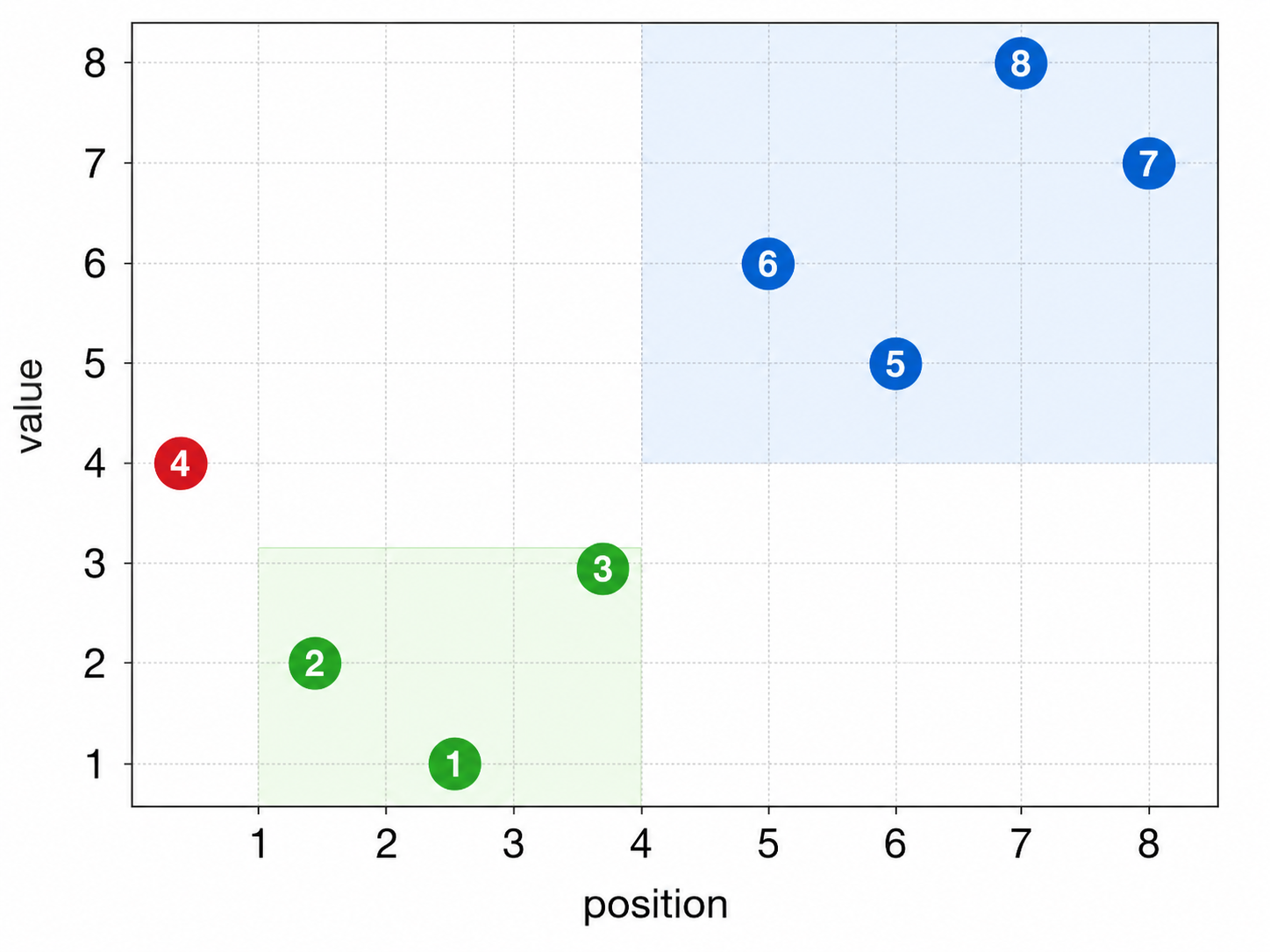}
    \end{minipage}
    \caption{
        Left: labeled tree decomposition of
        $\pi=(4,2,6,1,5,3,8,7)$. Cf.\ the identity permutation, which
        decomposes into a path. The left and right subtrees form fresh
        embedding problems, yielding recursive structure. Right: a graph of
        $\pi_{231}$, which has convenient block structure and a tree
        decomposition isomorphic to that of $\pi$.
    }
    \label{fig:perm-structure}
\end{figure}

\begin{theorem}\label{thm:main-fixed}
    Let $\pi \in \Pi_k$ and let $\alg^*\in\mathrm{OA}(\pi)$ denote \cref{alg:canonical-general2}. Then
    \begin{equation}\label{eq:beta-achieved}
        \E{T_{\alg^*}}=\beta(\pi),
    \end{equation}
    and, furthermore, for some universal constant $C > 0$,
    \begin{equation}\label{eq:fixed-prob-UB}
        \PP[\stream]{|T_{\alg^*} - \beta(\pi)| > y \,\beta(\pi)^{3/4}} \le C\,y^{-2}\,, \qquad y \ge 1 \,. 
    \end{equation}
    Conversely, for any $\alg \in \mathrm{OA}(\pi)$ and any $\delta \in (0,1)$, 
    \begin{equation}\label{eq:fixed-prob-LB}
        \PP[\stream]{T_\alg \le (1-\delta)\beta(\pi) } \le 1-\delta\,.
    \end{equation}
\end{theorem}

As an aside, Algorithm~\ref{alg:canonical-general2} does not use internal randomization, hence \eqref{eq:beta-achieved} implies that allowing internal randomization for algorithms does not speed up the optimal embedding.

\subsection{Scaling Laws} Our second main contribution concerns the asymptotics of $\beta(\pi)$ for
canonical families of permutations. The motivation comes from the classical problem in combinatorics of constructing short universal permutations. A permutation $\Sigma\in\Pi_n$ is $k$-universal if every $\pi\in\Pi_k$ appears as an order-isomorphic subsequence of $\Sigma$. The problem of determining the minimum length $n=n(k)$ of a $k$-universal permutation---posed in various forms since the 1970s by Knuth, Chung--Diaconis--Graham, Arratia, and others; see the survey \cite{engen-vatter}---remains open, despite considerable effort. The best known asymptotic bounds \cite{offline-lb,engen-vatter,miller2009} are
\[
    \frac{1.00007}{e^2}\,k^2 \le  n \le \frac{1 + o(1)}{2}\, k^2\,,
\]
with the lower bound due to the recent advance in \cite{offline-lb}.

A parallel question, attributed to Alon (1999) \cite{arratia}, asks for
which $n=n(k)$ is a uniformly random $n$-permutation $k$-universal with high probability. The identity permutation already forces $n\ge(1/4+o(1))k^2$ by the previously mentioned result of Logan–Shepp and Vershik–Kerov on offline LIS, and Alon conjectured
that $n=(1/4+o(1))k^2$ is indeed the threshold; in other words, monotone patterns should be the first-order obstruction to universality. Surprisingly, even less is known about the random-host setting than the deterministic setting. The best lower bound
is the aforementioned $k^2/4$ from offline LIS, while the best upper bound is order
$k^2\log\log k$, established in a recent breakthrough of He and Kwan \cite{he-kwan}. Moreover, the simpler problem, reiterated in \cite[Conjecture 5.3]{he-kwan}, of determining whether each fixed $\pi \in \Pi_k$ appears with high probability in a random permutation of length $(1/4+o(1))k^2$ remains open.

Our study of $\beta$ provides an online setup in which offline analogues
can be tested directly. Define
\[
    \beta_k^- := \min_{\pi\in \Pi_k}\beta(\pi),
    \qquad
    \beta^\typ_k := \E[\pi\sim \mathrm{Unif}(\Pi_k)]{\beta(\pi)}, \qquad  \beta_k^+ := \max_{\pi\in \Pi_k}\beta(\pi).
\]
Thus $\beta_k^+$ captures the online analogue of the most difficult target pattern to find in
a random host, while $\beta_k^\typ$ measures the embedding time of a typical
target pattern. The fastest target pattern to embed, captured by $\beta_k^-$, offers a new and natural extremal quantity as well.

\begin{question}\label{q:leading}
        Is the increasing sequence $\pi=\mathrm{Id}_k$ the hardest target pattern to
    embed, up to leading order? Equivalently, is
    \[
        \beta_k^+ = \frac{k^2}{2}(1+o_k(1))?
    \]
\end{question}
This is the online analogue of a pattern-wise weakening of Alon's conjecture: it
concerns the embedding time of individual patterns, rather than simultaneous
universal containment. 
Our second main result resolves this question and reveals a
separation from the conjectural offline picture.
\begin{theorem}[Scaling limits]\label{thm:numerical-scaling}
The following limits exist:
\[
    c_- := \lim_{k\to\infty}\frac{\beta_k^-}{k^2}, \qquad c_\typ:=\lim_{k\to\infty}\frac{\beta_k^\typ}{k^2},
    \qquad
    c_+:=\lim_{k\to\infty}\frac{\beta_k^+}{k^2}.
\]
Moreover,
\[
    c_- < c_\typ < \frac12 < c_+.
\]
Quantitatively, 
\[
    \frac{1}{4} \le c_- \le 0.48867, \qquad 0.48934 \le c_\typ \le 0.49967,
    \qquad
    0.50547 \le c_+ \le 0.50568.
\]
\end{theorem}
In particular, we obtain two striking negative results in the online setting. First, increasing permutations are not first-order extremal. This creates a dichotomy: either the online and offline permutation embedding problems have fundamentally different extremizers, or else Alon's conjecture and its weakening \cite[Conjecture 5.3]{he-kwan} do not hold. Both resolutions are interesting. Second, random permutations are also not first-order extremal. This runs counter to the
heuristic---which holds in the settings of universal metric spaces \cite{metricpoincare,matousek}---that random objects are
the hardest to embed. A similar phenomenon occurs in universal graphs \cite{alon2017}, where cliques and independent sets turn out to be strictly harder to embed than random subgraphs.

\begin{remark}
    We provide efficient algorithms for computing $c_+$ and $c_\typ$ to arbitrary precision in terms of $\beta^+_k$ and $\beta_k^\typ$, respectively, in \cref{lem:subsupadd}. Our methods allow for efficient, rigorous computation of $\beta_k^\pm$ to arbitrary precision. However, while $\beta_k^\typ$ can be computed to arbitrary precision (by enumerating over $\Pi_k$), we do not have an efficient routine, so our quantitative bounds on $c_\typ$ use indirect estimates (coming from the outer inequalities of \eqref{eq:betagamma}).
\end{remark}

\medskip

Our final result complements the above analysis of $c_\typ$ with the observation that the expected embedding time of a typical (random) permutation is concentrated around its mean, denoted by $\beta_k^\typ$, and even follows a central limit theorem.

\begin{theorem}[Concentration and CLT for $\beta$]\label{thm:main-random} 
Letting $Z$ be a standard normal and $\boldsymbol{\pi} \sim \mathrm{Unif}(\Pi_k)$, there exists a universal constant $\sigma^2 > 0$ such that as $k$ diverges,
\[
    \frac{\beta(\boldsymbol{\pi}) - \beta_k^\typ}{k^{3/2}} ~\stackrel{\mathcal{D}}{\longrightarrow}~ \sigma Z\,.
\]
In particular, for every $\varepsilon>0$,
\begin{equation}\label{eq:goal-beta-concentration}
    \frac{\ba{\cb{\pi \in \Pi_k \,:\, \beta(\pi) \in
    [c_{\mathrm{typ}}\,k^2 - \varepsilon k^2,
    c_{\mathrm{typ}}\,k^2 + \varepsilon k^2]}}}{|\Pi_k|} = 1 - o(1)\,.
\end{equation} 
\end{theorem}

\subsection{Open problems} We conclude with some open problems. The first asks whether the coarse probabilistic threshold in \cref{thm:main-fixed} can be sharpened, so that $\beta(\pi)$ is the true probabilistic sharp threshold for online embedding. 

\begin{question}[Sharp lower threshold]
Do there exist vanishing sequences \(\delta_k,\delta'_k\searrow 0\) such that, for every
\(\pi\in\Pi_k\) and every \(\alg\in\mathrm{OA}(\pi)\),
\[
    \PP[\stream]{T_\alg < (1-\delta_k)\beta(\pi)}
    <
    \delta'_k?
\]
\end{question}
\noindent A weaker variant asks whether the same conclusion holds for \(\mu_k\)-asymptotically almost every \(\pi\). 

Our next problem concerns the minimum possible value of $\beta$. 
Based on compelling (but non-rigorous) numerical evidence, we believe that $\beta_k^-$ is saturated by the ``alternating increasing\footnote{Here, ``increasing'' is inserted to distinguish from the terminology ``alternating sequence'', which is used in the literature to describe any sequence $(a_1,a_2,\dots)$ satisfying the weaker condition: $a_1 < a_2$, $a_2 > a_3$, $a_3 < a_4$, and so on.} sequence,'' defined as follows.

\begin{definition}[Alternating increasing sequence]\label{def:alternating-increasing}
For $m \ge 1$, define
\begin{align*}
    \mathrm{Alt}_{2m} &:=(2,1,4,3,\ldots,2m,2m-1)\in\Pi_{2m} \,,\\
    \mathrm{Alt}_{2m+1} &:=(2,1,4,3,\ldots,2m,2m-1,2m+1)\in\Pi_{2m+1} \,.
\end{align*}
By convention, set $\mathrm{Alt}_{1} := (1)$. Also, provided the following limit exists, define
\[
    c_{\mathrm{alt}}
    :=
    \lim_{n\to\infty}\frac{\beta(\mathrm{Alt}_n)}{n^2}\,.
\]
\end{definition}

Identifying $\beta^-_k$ with $\beta(\mathrm{Alt}_k)$ would give an explicit, combinatorial description of the easiest-to-embed permutations.  Moreover, $c_{\mathrm{alt}}$ should be tractable to analyze using existing methods from online LIS, as the alternating increasing sequence has roughly the same self-similarity and renewal structure as the identity permutation. We believe the following (stated in increasing order of difficulty).

\begin{conjecture}[Minimizer of $\beta$]\leavevmode

\begin{enumerate}
    \item The limit $c_{\mathrm{alt}}$ exists and admits the explicit variational formula:
    \[
        c_{\mathrm{alt}} = \min_{x \in (0,1)} \frac{(1 + x)(1 - \log x)}{8(1 - x)} \approx .4883.
    \]
    \item The alternating increasing sequence is the asymptotic minimizer of $\beta$:
    \begin{equation}\label{eq:conj cminus}
        c_- = c_{\mathrm{alt}}\,.
    \end{equation}
    \item For every $k \in \NN$, the alternating increasing sequence is a minimizer of $\beta$:
    \[
        \mathrm{Alt}_k \in \argmin_{\pi \in \Pi_k} \beta(\pi)\,.
    \]
\end{enumerate}
\end{conjecture}

Our next question asks for a similarly explicit description of the maximizers of
\(\beta\).
Although \eqref{eq:beta+} readily yields a dynamic program for $\beta_k^+$ that also gives constructions of maximizers for any fixed \(k\), it is unclear whether the resulting examples belong to a simple family. Our constructions are neither near-monotone nor nearly uniformly random, but they may have a recursive structure when viewed through their associated tree shapes.

\begin{question}
    For all sufficiently large (but finite) $k$, classify 
    \[
        \argmax_{\pi \in \Pi_k} \beta(\pi).
    \]
    Do these extremizers admit a simple description (perhaps in terms of their tree representations)?
\end{question}

In connection with prophet inequalities, we are also interested in the largest gap between online and offline embedding times. Formally, fix $\pi \in \Pi_k$, let $\sigma \sim \mathrm{Unif}(\Pi_n)$, and define $n_c(\pi)$ as the minimum $n$ such that $\PP{\sigma \text{ contains } \pi}\geq 1/2$.

\begin{problem}[Extremal online/offline gap]
    Determine
    \[g:=\limsup_{k \rightarrow \infty} \max_{\pi \in \Pi_k} \frac{\beta(\pi)}{n_c(\pi)}.\]
\end{problem}
For example,  considering $\pi = \mathrm{Id}$ gives $g \ge 2$. The next open question is related to universality in the online setting, i.e., the task of embedding all permutations simultaneously.

\begin{question}[Online universality]
For a fixed $k$,
consider a family of online algorithms $\mathrm{OA}(\Pi_k)$,
with each algorithm in $\mathrm{OA}(\Pi_k)$ producing
online order-isomorphic embeddings of all $k!$ permutations from $\Pi_k$
into the stream $\stream$. For every $\alg\in \mathrm{OA}(\Pi_k)$,
let $T_\alg$ be the (random stopping) time of embedding all elements of $\Pi_k$. What are the asymptotics of 
$$\frac{1}{k^2}
\inf\limits_{\alg\in \mathrm{OA}(\Pi_k)}\,\EE\,T_\alg?$$
\end{question}

In contrast with the offline setting, we suspect this quantity diverges with $k$. 
Our final question is open ended and motivated by the optimal-stopping aspect of the model.
What remains of the theory when the iid stream is replaced by a stream with strong dependencies? The aim here is conceptual rather than tied to any particular examples:

\begin{question}[Structured host] Replace $\stream$ with a stochastic process, such as (discrete-time subsamplings of) the Ornstein-Uhlenbeck process, or a reflected Brownian motion within an interval. Does the optimal expected embedding time of a fixed permutation still admit a tractable representation? 
\end{question}

We record a basic partial result in this direction: consider the finite-host variant in which the stream is replaced by a uniformly random host permutation. For $\pi \in \Pi_k$, let $\hat{\beta}(\pi)$ be the minimum $n$ such that there is an online algorithm for embedding $\pi$ in a sequentially revealed random permutation of length $n$ that succeeds with probability at least $1/2$. An elegant coupling introduced in \cite{permutation} shows that for the online LIS problem, embedding in a stream is at least as difficult as embedding in a random permutation. This same coupling readily extends to our setting of embedding general permutations, giving the analogous result: 

\begin{proposition}\label{prop:Stopi}
    For any $\pi \in \Pi_k$, 
    \[\hat \beta(\pi) \le (1 + o_k(1))\,\beta(\pi).\] 
\end{proposition}
\noindent Here, the error $o_k(1)$ is uniform over $\pi \in \Pi_k$. The proof is given in the appendix.

\subsection{Organization} \Cref{sec:ov} fixes notation and conventions, and then gives a technical overview of the main proofs. \Cref{sec:recursion} proves the dynamic-programming recurrence \eqref{eq:rec fixed}, constructs the optimal recursive algorithm used in \eqref{eq:beta-achieved}, and establishes the lower bound \eqref{eq:fixed-prob-LB}. \Cref{sec:prob} proves the upper-tail guarantee \eqref{eq:fixed-prob-UB}, and then proves the typical-pattern concentration and central limit theorem in \cref{thm:main-random}. \Cref{sec:scaling} begins the proof of the scaling limits in \cref{thm:numerical-scaling} by using dynamic programming and sub/super-additivity to reduce the theorem to estimation of finite auxiliary sequences. \Cref{sec:numer} uses rigorous interval arithmetic to estimate these sequences, establishing \cref{thm:numerical-scaling}; the accompanying coding supplement is recorded in \url{https://github.com/dylanaltschuler/online-permutation-embedding}.

\section{Technical Overview}\label{sec:ov}
\subsection{Notation and conventions}\label{sec:ov-notation}
Before surveying the proofs of the main results, we collect key preliminaries. 
Let $[k] := \{1,2,\dots,k\}$ and $\NN := \{1,2,\dots\}$. Denote the set of permutations on $[k]$ by $\Pi_k$ with the convention $\Pi_0 = \{\emptyset\}$. Let $\mu_k$ denote the uniform measure on $\Pi_k$. Throughout this article, $k \in \{0,1,2,\dots\}$.
\begin{itemize}
    \item \textbf{Stream}. Let $\stream = (X_i)_{i\in \NN}$ denote a sequence of iid draws from the uniform distribution on $[0,1]$. The natural and augmented filtrations are, respectively,
    \[
        \mathcal{F}_t^\circ := \sigma(X_1,\dots,X_t),
        \qquad
        \mathcal{F}_t := \sigma(X_0,X_1,\dots,X_t),
        \qquad t\ge0,
    \]
    where an empty list generates the trivial $\sigma$-field and $X_0$ is a distinguished random variable independent of $\stream$, called the ``seed,'' representing possible internal randomization of algorithms.
    \item \textbf{Auxiliary randomness}. For concreteness, we regard $X_0$ as a random element of a standard Borel space $\mathcal X_0$ rich enough to encode a countable sequence of independent auxiliary random variables. Conditioning on the seed and fixing a favorable seed value shows that access to internal randomness does not change the optimal expected run-time. It nevertheless enables natural coupling arguments that are convenient for analysis.

    \item \textbf{Algorithm}. An online embedding algorithm for a permutation $\pi$ is a sequence of stopping times $(\tau_1,\dots,\tau_k)$ adapted to $(\mathcal F_t^\circ)_{t\ge0}$, almost surely satisfying both $1\le \tau_1<\cdots<\tau_k$ and $(X_{\tau_1},\dots,X_{\tau_k}) \simeq \pi$. When auxiliary randomness is explicitly allowed, the stopping times are instead adapted to $(\mathcal F_t)_{t\ge0}$. For an embedding algorithm $\alg = (\tau_1,\dots,\tau_k)$, we write $T_\alg := \tau_k$ to denote the run-time. 

    \item \textbf{Stopping time}. We often use the following deterministic representation of stopping times. If $\tau_j$ is adapted to $(\mathcal F_t)_{t\ge0}$, then, after modifying on null sets if necessary, there are Borel functions
    \[
        g_j^{(\ell)}:{\mathcal X_0}\times [0,1]^\ell\to\{0,1\}\,,
    \]
    where
    \[
        \tau_j=\min\big\{\ell\geq 1:\;g_j^{(\ell)}(X_0,X_1,\dots,X_\ell)=1\big\},\quad 1\leq j\leq k.
    \]
    In the non-random case, the same representation holds with the
    $\mathcal X_0$-coordinate omitted.

    \item \textbf{Left and right subpermutations}.  
    Let $\pi\in\Pi_k$. We define the left and right permutations of $\pi$, denoted by $L := L(\pi)$ and $R:=R(\pi)$, as the unique permutations satisfying
    \begin{align*}
        L \in \Pi_{\pi(1)-1}\,, \quad L &\simeq (\pi(j)\,:\, \pi(j) < \pi(1)) \, \\
        R \in \Pi_{k-\pi(1)}\,, \quad R &\simeq (\pi(j)\,:\, \pi(j) > \pi(1))  \,.
    \end{align*}
    That is, $L$ and $R$ are the permutations induced by $\pi$, restricted to the ranges $[\pi(1)-1]$ and $[k] \setminus [\pi(1)]$, respectively.
    \item \textbf{Tree representation.}
For $\pi\in\Pi_k$, let $\tree_\pi$ be the binary search tree
obtained by inserting the keys
$\pi(1),\pi(2),\ldots,\pi(k)$
in order, with each node labeled by its insertion time. More explicitly, $\tree_\pi$ is the labeled tree with root labeled 1, label-set given by $[k]$, and all nodes having at most two children, constructed top-down as follows. For a node labeled $m$, the left and right children of $m$ are, respectively:
\begin{align*}
            &\min\cb{ \ell > m :
            \pi(\ell) < \pi(m)
            \text{ and }
            [\pi(\ell),\pi(m)]
            \cap
            \cb{\pi(s): s < m}
            =
            \varnothing}\,, \\
            &\min\cb{ r > m :
            \pi(r) > \pi(m)
            \text{ and }
            [\pi(m),\pi(r)]
            \cap
            \cb{\pi(s): s < m}
            =
            \varnothing}\,.
    \end{align*}
    Note that these sets can be empty, in which case $m$ can have degree one or zero.
    We refer to Figure~\ref{fig:perm-structure} for an illustration.
    As shorthand, we denote the left and right maximal subtrees of the root by $\tree_L := \tree_{L(\pi)}$ and $\tree_R := \tree_{R(\pi)}$. We also adopt the conventions:
\[
    \beta(\tree_\pi) := \beta(\pi)\,,  \quad \text{ and } \quad  \beta(\emptyset) = 0\,.
\]
    \item \textbf{Subtree and fringe tree}. By convention, a ``subtree'' rooted at $v$ always refers to the maximal subtree (also called the \textit{fringe subtree}), consisting of the root $v$ along with all descendants of $v$ in the parent tree.
    \item \textbf{Recurrence kernel and empty subpermutation convention}. Define the extended-value function on $[0,1]$
    \begin{equation}\label{eq: phipq definition}
    \phi_{p,q}(x)
    :=
    \frac{p}{x}\ind_{\{p>0\}}
    +
    \frac{q}{1-x}\ind_{\{q>0\}}.
    \end{equation}
By convention, set $\frac{p}{x}\ind_{\{p>0\}}$ to zero if $p=x=0$, and to $+\infty$ if $x=0$ when $p\neq0$; a symmetric convention is adopted for the other term. We refer to $\phi$ as a ``continuation cost,'' as $\phi_{p,q}(x)$ gives the expected cost of the two recursive subproblems when their costs are $p$ and $q$ and the root is embedded at relative location $x$.

Next, for $0\le a<b\le1$, set
\begin{equation}\label{eq: Gpq definition}
    G_{p,q}(a,b)
    :=
    \frac{
        1+\int_a^b \phi_{p,q}(x)\,dx
    }{b-a},
\end{equation}
where the integral is understood as an improper integral and can take the value $+\infty$. We further define for $p,q \ge 0$:
\begin{equation}\label{eq:def-G1}
G(p,q)
:=
\min_{0 \le a < b \le 1} G_{p,q}(a,b) = \min_{0 \le a < b \le 1}
\frac{1+p\log(b/a)+q\log((1-a)/(1-b))}{b-a}\,.
\end{equation}
The logarithmic representation is valid as written for $p,q > 0$, and should be understood with the same convention of omitting terms when $p$ or $q$ is zero. In particular, the Bellman equation \eqref{eq:rec fixed} is: $$\beta(\tree) = G(\beta(\tree_L),\beta(\tree_R)).$$ Moreover, the convention $\beta(\emptyset)/0 := 0$ in \eqref{eq:rec fixed} is implied by the analogous conventions for $\phi_{p,q}$ and $G(p,q)$.
\end{itemize}

We now give a technical overview of the rest of the paper.
\subsection{Bellman equation for \texorpdfstring{$\beta(\pi)$}{β(π)}} For a finite rooted binary tree $S$, define $b(S)$ recursively by
\[
    b(\emptyset):=0,
    \qquad
    b(S):=G\big(b(S_L),b(S_R)\big),
\]
where $S_L$ and $S_R$ denote left and right fringe tree, respectively,
obtained by removing the root of $S$
(so that $S_L$ and/or $S_R$ can be empty)\footnote{We assume that each child node within the tree is labeled either ``left'' or ``right'', so that no ambiguity in the recursive decomposition of $S$ into subtrees arises.}.
Our goal is to show $b(S) = \beta(S)$. 

The upper bound $\beta(\pi) \le b(\tree_\pi)$ follows from considering the embedding algorithm naturally implied by \eqref{eq:rec fixed}, namely \Cref{alg:canonical-general2}. When a subtree $S$ is to be embedded in an interval of length $d$, the algorithm waits until the stream enters the rescaled optimizing interval for $G(b(S_L),b(S_R))$. Conditional on the accepted relative location $x$, the lower and upper thinned streams are fresh uniform streams, and the expected active running times of the two recursive calls are $b(S_L)/(dx)$ and $b(S_R)/(d(1-x))$. Averaging over the optimizing interval gives total expected running time $b(S)/d$. This construction is formalized as \cref{alg:canonical-general2} and, \textit{a posteriori}, is an optimal embedding.

The matching lower bound develops a potential function argument that also yields the probabilistic lower bound \eqref{eq:fixed-prob-LB}. This argument combines analysis of our specific potential function with several existing techniques from the optimal stopping theory. The main idea is to study the expected ``continuation cost'' of embedding the remainder of a permutation, once a prefix is embedded. The surplus continuation cost for an arbitrary algorithm forms a submartingale.
On the other hand, under the optimal algorithm, the corresponding process is a martingale. Optional stopping allows for bounding the surplus.
This is related to Snell-envelope arguments \cite{Snell} in 
optional stopping literature; we also refer to \cite{arlotto-nguyen-steele} for related constructions.

In more detail, given a valid partial embedding ${\bf x}\in(0,1)^j$ for the $j$--prefix of $\pi$, delete the labeled vertices from $\tree_\pi$ corresponding to the prefix. The remaining components form a collection $\calC({\bf x})$ of fringe subtrees $S$, each of which is constrained to be embedded into certain interval $I_S({\bf x})\subset(0,1)$ (see definition \eqref{eq: interval IS definition} further in the text). Define the potential
\[
    V({\bf x})
    :=
    \sum_{S\in\calC({\bf x})}
    \frac{b(S)}{|I_S({\bf x})|}.
\]
Here, $V$ represents the aforementioned continuation cost, in other words, $V$
corresponds to optimal embedding of the rest of the permutation $\pi$.
We then make the key claim:
\begin{equation}\label{eq:ov-optimality}
        \int_0^1
    \big(V({\bf x})-V({\bf x}\oplus z)\big)_+\,dz
    =1.
\end{equation}
The proof is analytic and developed in \cref{prop:F-root}.
Intuitively speaking, \eqref{eq:ov-optimality} expresses Bellman's principle: an optimal strategy should produce embedding of any prefix in an optimal manner. Next, for an arbitrary embedding algorithm, if ${\bf X}_t$ is its selected prefix by time $t$, letting
\[
    W_t:=\min\{V({\bf X}_t),b(\tree_\pi)\},
    \qquad
    M_t:=t+W_t,
\]
it is readily deduced from \eqref{eq:ov-optimality} that $(M_t)$ is a submartingale. Since $W_0=b(\tree_\pi)$ and $W_{T_\alg}=0$, optional stopping gives $\E{T_\alg}\ge b(\tree_\pi)$. Together with the recursive upper bound, this proves \eqref{eq:rec fixed} and the optimality of \cref{alg:canonical-general2}. The same submartingale also incidentally gives the lower threshold \eqref{eq:fixed-prob-LB}. 

\subsection{Probability bounds}

Once the recurrence \eqref{eq:rec fixed} is established, our next two goals are to prove the upper threshold for the canonical algorithm and to investigate $\beta(\pi)$ for typical (i.e., random) permutations $\pi$. The lower threshold was proved by the submartingale argument above.

\subsubsection{Upper threshold}
We aim to establish \eqref{eq:fixed-prob-UB}, which states that the embedding time of
$\pi$ under Algorithm~\ref{alg:canonical-general2} is concentrated about its
mean $\beta(\pi)$. The proof first develops a Bellman equation that upper bounds the variance of $T_{\alg^*}$, obtained from using the law of total variance to condition on $X_{\tau_1}$ and the number of future arrivals that fall into $[0,X_{\tau_1}]$ and $[X_{\tau_1},1]$ (i.e., the lower and upper future streams). 

The next task is to extract from this recurrence the global bound that uniformly for all $\pi \in \Pi_k$,
\[
    \pVar{T_{\alg^*}} \le C\,\beta(\pi)^{3/2}.
\]
Note that \eqref{eq:fixed-prob-UB} would then follow via Chebyshev's inequality. The principal concern is that fluctuations of $X_{\tau_1}$ might later force a future portion of the embedding into an interval that is too small, causing a disproportionate
increase in the conditional variance of the remaining embedding time. The reason this
does not happen is that such an occurrence would also cause a disproportionate increase in the conditional expectation of the remaining embedding time, so that the optimal choice of $(a,b)$ in \eqref{eq:rec fixed} necessarily selects for low variance of the remaining embedding time. This ``self-balancing'' is made precise through a calculus computation that uses the first-order stationary conditions in \eqref{eq:rec fixed} to estimate the variance of the optimal algorithm in terms of its expectation. This bootstrapping approach leads to the claimed variance bound.

\subsubsection{Expected embedding time of typical patterns}
Lastly, we address \cref{thm:main-random} which seeks to understand the distribution of
$\beta(\pi)$ for $\pi\sim\mathrm{Unif}(\Pi_k)$.
The key observation is that the distribution of $\tree_\pi$  precisely coincides with a classical model of random binary search tree (BST) for which concentration results are readily available. More precisely, results are readily available for linear functionals of the fringe subtrees of a random BST:
\[
    \sum_{v \in V(\tree)} \eta(\tree_v)\,,
\]
where $\eta$ is a functional mapping trees to $\RR$. Typically, uniform bounds on $\eta$ are needed to apply existing results (since this enables construction of certain martingales with bounded increments). 

Our proof proceeds by defining $\eta$ via:
\[
    \sqrt{\beta(\tree)} = \sqrt{\beta(\tree_L)} + \sqrt{\beta(\tree_R)} + \eta(\tree),
\]
from which a simple recursion gives:
\[
    \sqrt{\beta(\tree)}
    =
    \sum_{v\in V(\tree)} \eta(\tree_v).
\]
The choice to work with $\sqrt{\beta}$, rather than directly with $\beta$, emerges from analyzing the first-order stationary conditions for $G(p,q)$ in \eqref{eq:def-G1}.
A calculus computation establishes that $\eta(\tree) \in [0,1]$, allowing us to apply a central limit theorem of Holmgren and Janson \cite{tree-CLT} to obtain a normal limit for $\sqrt{\beta}$ and, via the delta method, for $\beta$. 

A consequence of invoking such a general tool is that the variance parameter in the CLT is non-explicit and, in particular, not guaranteed to be non-zero. Hence, a final substantial argument is needed to pin down the exact asymptotic scale of the fluctuations, and in particular to rule out ``superconcentration'' phenomena. The argument proceeds inductively based on $k$. As a base case, an elementary computation shows that $\beta(\pi)$ has non-zero variance for $\pi \sim \mu_3$. The inductive step utilizes the recurrence \eqref{eq:rec fixed} 
to prove that for any tree $\tree$, variance from $\beta(\tree_L)^{1/2}$ and $\beta(\tree_R)^{1/2}$ contribute non-negligible variance to $\beta(\tree)^{1/2}$. Making this argument quantitative yields an inductive variance lower bound that matches the order of growth of the upper bound from the Holmgren--Janson result. 

\subsection{Scaling laws}
The proof of \cref{thm:numerical-scaling} proceeds in two stages. First, in \cref{sec:scaling}, these asymptotic estimates are reduced to finite computations. Then, in \cref{sec:numer}, interval arithmetic is used to complete these finite computations with rigorous error guarantees.

\subsubsection{Reduction to finite computations} We
consider
\[
    \beta_k^+
    :=
    \max_{\pi\in\Pi_k}\beta(\pi),
    \qquad
    \beta_k^-
    :=
    \min_{\pi\in\Pi_k}\beta(\pi),
    \qquad
    \beta_k^\typ
    :=
    \E[\pi\sim\mathrm{Unif}(\Pi_k)]{\beta(\pi)}.
\]
We begin by observing that the recurrence \eqref{eq:rec fixed}, combined with the fact that $G$ is monotone in both of its arguments, readily yields dynamic programs for the
extremal sequences:
\[
    \beta_k^+
    =
    \max_{i\in[k]}G(\beta_{i-1}^+,\beta_{k-i}^+),
    \qquad
    \beta_k^-
    =
    \min_{i\in[k]}G(\beta_{i-1}^-,\beta_{k-i}^-).
\]
The typical sequence $\beta_k^\typ$ does not admit such a recurrence, since
expectation does not commute with the minimization in \eqref{eq:rec fixed}. However, interchanging expectation and the minimization does yield an inequality: $\beta_k^\typ\le \gamma_k$ 
where $\gamma_k$ is the tractable auxiliary sequence
\[
    \gamma_0=0,\qquad
    \gamma_1=1,\qquad
    \gamma_k
    =
    \frac1k\sum_{i=1}^k G(\gamma_{i-1},\gamma_{k-i}).
\]

We next prove the existence of the scaling limits $c_-$, $c_+$, and $c_\typ$ via subadditivity. Subadditivity arises from a
natural class of constrained embedding strategies. Fix \(\ell \in [k]\) and \(\alpha\in(0,1)\), and enforce the constraints that $\{\pi(j)\,:\,\pi(j) \le \ell\}$ is embedded in $[0,\alpha]$ and $\{\pi(j)\,:\,\pi(j) > \ell\}$ is embedded in $[\alpha,1]$. This separates the embedding task into two independent subproblems, yielding the desired subadditivity structure for $\beta^\pm_k/k$ and $\beta^\typ_k/k$. 

Fekete's lemma then both establishes the existence of
the corresponding limits and also suggests a procedure for generating asymptotic bounds via finite computations. Considering $c_+$ for example, for any $k_0 \in \NN$:
\[
    c_+ := \lim_{k \to \infty} \frac{\beta_k^+}{k^2} = \inf_{k \ge 1} \frac{\beta_k^+}{k^2} \le \frac{\beta_{k_0}^+}{k_0^2}
\]
A finite-computation lower bound on $c_+$ is symmetrically obtained after proving that the shifted sequence $(\beta_{k-1}^+)^{1/2}$ is \textit{super}additive. For lower bounding $c_-$, no obvious lower envelope is available, and we instead prove the analytic uniform bound $\beta(\pi) \ge k^2/4$ for all $\pi \in \Pi_k$ via calculus estimates on $G$. On the other hand, we \textit{can} rigorously efficiently estimate the finite sequence $\beta_k^-$; while this does not immediately imply lower bounds for $c_-$, it actually lets us prove a useful lower bound for $c_\typ$ via \eqref{eq:betagamma}.

In total, the proof of \cref{thm:numerical-scaling} now reduces to rigorous
estimates of \(\gamma_k\), \(\beta_k^+\), and $\beta_k^-$ for large \(k\).

\subsubsection{Interval arithmetic}

We now turn to the rigorous estimation $\gamma_k$ and $\beta_k^\pm$, carried out in \cref{sec:numer}. For brevity, let $\xi_k$ denote one of $\{\gamma_k, \beta_k^-,\beta_k^+\}$ as all are handled similarly. Then $\xi_k$ is expressed through a dynamic program involving basic arithmetic, maxima and minima of finite collections of scalars, and repeated evaluation of the kernel
\[
    G(p,q)
    =
    \min_{0\le a<b\le1}
    \frac{
        1+p\log(b/a)+q\log((1-a)/(1-b))
    }{b-a}.
\]
The key structural fact is again that \(G\) is monotone in both inputs. Hence if $p$ and $q$ are rigorously known to lie in some intervals denoted by $[p] := [\underline p, \overline p]$ and $[q] := [\underline q, \overline q]$, then $G(p,q) \in [G]([p],[q])$ where
\begin{equation}\label{eq:overview-G}
    [G]([p],[q])
    :=
    [G(\underline p,\underline q),G(\overline p,\overline q)]
\end{equation}
So, once the $(\xi_j)_{j < k}$ are enclosed in intervals,
monotonicity propagates these enclosures through sums, maxima, and minima, yielding enclosing intervals for $\xi_k$. 

It remains to certify evaluations of $G(\underline p,\underline q)$ and $G(\overline p,\overline q)$, as they involve a non-elementary minimization problem. Using the first-order optimality conditions, our proof first reduces this minimization to the one-dimensional problem of finding the unique root of some ``residual'' function \(\mathsf I_{p,q}:\RR \to \RR\). We then prove that for any interval $[u]$, there is a corresponding interval $[G_{[u]}]$ with
\[
    \mathsf I_{p,q}(\underline u) \le 0 \le \mathsf I_{p,q}(\overline u) \Longrightarrow G(p,q) \in [G_{[u]}]\,.
\]
This is carried out in Lemma~\ref{lem:G-cert} and is the main mathematical content of \cref{sec:numer}. Crucially, $\mathsf I_{p,q}$ is a composition of elementary functions\footnote{That is, functions which can be computed to arbitrary floating point precision with rigorous guarantees using standard numerical packages; e.g., arithmetic, maxima or minima of finite collections of scalars, hyperbolic trigonometric functions.}. Hence, for any interval $[u]$ we can rigorously verify claims of the form ``$[G]([p],[q]) \subset [G_{[u]}]$'' by computing the lower and upper ``slacks''
\[
    (\mathsf I_{\underline p, \underline q}(\underline u))_-, \quad (\mathsf I_{\overline p, \overline q}(\overline u))_+ 
\]
and checking that they are non-zero and larger in magnitude than any floating point error incurred by the elementary operations in computing $\mathsf I_{p,q}$. Note that these certificates are ``sound,'' not ``complete'': successful sign checks imply the claimed enclosure, but our procedure need not validate every true enclosure.

In summary, the rigorous numerical computation
starts from the exact boundary conditions \(\xi_0\in[0,0]\) and
\(\xi_1\in[1,1]\), then inductively certifies interval enclosures for
\(\xi_k\) using the previously certified enclosures. For the purposes of verifying \cref{thm:numerical-scaling}, considering $k$ up to $k_0 = 5000$ suffices. The candidate intervals,
verification script, and certified output are supplied in the computational
supplement.

\section{Recursive formula and algorithm for \texorpdfstring{$\beta(\pi)$}{β(π)}}\label{sec:recursion}

In this section, we prove the recurrence~\eqref{eq:rec fixed}. For the reader's
convenience, we restate it here formally as a theorem in terms of the notation introduced in
\cref{sec:ov-notation}.

\begin{theorem}[Recursion formula for $\beta$]
\label{thm:recursion-formula}
Let $\pi$ be a non-empty permutation. Then
\[
    \beta(\pi)
    =
    G\big(\beta(L(\pi)),\beta(R(\pi))\big),
\]
where $G$ is given by \eqref{eq:def-G1}.
\end{theorem}

We first study the structure of the minimization problem underlying $G$. An optimal stopping argument will then be used to prove that $\beta$ is lower bounded by $G$; a direct construction will supply the upper bound.  

\subsection{Minimization of \texorpdfstring{$G_{p,q}$}{G(p,q)}}

\begin{proposition}[Minimization of $G_{p,q}$]
\label{prop:F-root}
Fix $p,q \ge 0$, let $\phi := \phi_{p,q}$ be defined by \eqref{eq: phipq definition}, and, for $v\in\RR$, define $I_v:=\{x\in[0,1]\,:\,\phi(x)<v\}$ and
\[
    F(v):=\int_{I_v}(v-\phi(x))\,dx
    =
    \int_0^1 (v-\phi(x))_+\,dx.
\]
Then the following hold:
\begin{enumerate}
    \item For each $v\in\RR$, $I_v$ is a (possibly empty) interval subset of $[0,1]$.
    \item There exists a unique $v_*$ such that $F(v_*)=1$.
    \item Let $a_*:=\inf I_{v_*}$ and $b_*:=\sup I_{v_*}$ denote the endpoints of  $I_{v_*}$. Then $0\le a_*<b_*\le1$, and
    \begin{equation}\label{eq:v* min}
        v_*
        =
        G_{p,q}(a_*,b_*)
        =
        \min_{0\le a<b\le1}G_{p,q}(a,b),
    \end{equation}
where $G_{p,q}$ is given by \eqref{eq: Gpq definition}.
\end{enumerate}
\end{proposition}

\begin{proof}[Proof of Proposition~\ref{prop:F-root}]
The first claim is immediate upon noticing that $\phi$ is convex on its effective
domain, so that all of its strict sublevel sets are possibly empty intervals.

Towards the second claim, we begin with existence of $v_*$ via the intermediate
value theorem. Clearly $F$ is non-decreasing by construction. Moreover, $F$ is
$1$-Lipschitz. Indeed, for all $v,w\in\RR$, we have
\[
    \ba{(v-\phi(x))_+-(w-\phi(x))_+}\le \ba{v-w},
\]
and therefore $\ba{F(v)-F(w)}\le \ba{v-w}$. In particular, $F$ is continuous.

So, define $m:=\inf_{x\in[0,1]}\phi(x)$, and note that $m<\infty$; this holds, for example, because $\phi(1/2)<\infty$. Hence Lebesgue's dominated convergence theorem gives $F(v)\to 0$ as $v \searrow m$. On the other hand, $\phi$ is bounded on some
non-empty interval of positive Lebesgue measure; for instance, on any compact
subinterval of $(0,1)$. Hence $F(v)\to\infty$ as $v\to\infty$. The intermediate
value theorem therefore implies that $v_*$ exists.

Next, we prove uniqueness of $v_*$. For any $v_2>v_1$ with $F(v_1)>0$, it follows
that $I_{v_1}$ has positive Lebesgue measure. We then have
\[
    F(v_2)-F(v_1)
    \ge
    (v_2-v_1)\ba{I_{v_1}}>0.
\]
This implies $F$ is strictly increasing on the region where it is positive, and hence the equation $F(v)=1$ is attained by at most one value of $v$. Uniqueness is established.

Finally, we check that $(a_*,b_*)$ is the global minimizer of $G_{p,q}$. Since \(I_{v_*}\) is equal to \((a_*,b_*)\), up to possibly the inclusion of endpoints, the identity \(F(v_*)=1\) gives \[ 1+\int_{a_*}^{b_*}(\phi(x)-v_*)\,dx=0\,, \qquad \text{ hence } \qquad G_{p,q}(a_*,b_*) = \frac{1+\int_{a_*}^{b_*}\phi(x)\,dx}{b_*-a_*} = v_*\, . \]

Now let $0\le a<b\le1$ be arbitrary. If $\int_a^b\phi(x)\,dx=+\infty$, then
$G_{p,q}(a,b)=+\infty$ and there is nothing to prove. Otherwise, we compute:
\[
    \int_a^b(v_*-\phi(x))\,dx
    \le
    \int_0^1(v_*-\phi(x))_+\,dx
    =
    F(v_*)
    =
    1\,,
\]
from which we deduce
\[
    G_{p,q}(a,b)-v_*
    =
    \frac{
        1+\int_a^b(\phi(x)-v_*)\,dx
    }{b-a}
    =
    \frac{
        1-\int_a^b(v_*-\phi(x))\,dx
    }{b-a} \ge 0\,.
\]
In total,
\[
    G_{p,q}(a,b)\ge v_*=G_{p,q}(a_*,b_*).
\]
Since $0\le a<b\le1$ was arbitrary, the proof is complete.
\end{proof}

\subsection{Recursive equation and proof of \eqref{eq:fixed-prob-LB}}

For a finite rooted binary (maximum degree at most two) tree $S$, define $b(S)$ recursively by
\begin{equation}\label{eq:recursive-value-b}
    b(\emptyset):=0,
    \qquad
    b(S):=G\big(b(S_L),b(S_R)\big),
    \qquad S\neq\emptyset.
\end{equation}
Note that this uniquely defines $b(S)$, since $S_L$ and $S_R$ have fewer vertices than
$S$.

\begin{proposition}[Recursive lower bound]
\label{prop:recursive-value-lower-bound}
Let $\pi\in\Pi_k$ be non-empty. For every
$\alg\in\mathrm{OA}(\pi)$,
\[
    \E{T_\alg}\ge b(\tree_\pi).
\]
In particular, $\beta(\pi)\ge b(\tree_\pi)$. Moreover, for every integer $n\ge0$,
\begin{equation}\label{eq:finite-horizon-lower-tail}
    \PP[\stream]{T_\alg\le n}
    \le
    \frac{n}{b(\tree_\pi)}.
\end{equation}
\end{proposition}

\begin{proof}
First, we note that by \eqref{eq:recursive-value-b} and Proposition~\ref{prop:F-root}, every
non-empty finite binary tree $S$ satisfies
\begin{equation}\label{eq:recursive-value-calibration}
    \int_0^1
    \big(
        b(S)-\phi_{b(S_L),b(S_R)}(y)
    \big)_+\,dy
    =1.
\end{equation}

Set $T:=\tree_\pi$.
For $0\le j\le k$, let
\[
    \mathcal H_j
    :=
    \left\{
        {\bf x}=(x_1,\ldots,x_j)\in(0,1)^j:
        {\bf x}\simeq(\pi(1),\ldots,\pi(j))
    \right\}
\]
denote the set of valid $j$-prefix embeddings. Fix ${\bf x}\in\mathcal H_j$.
Delete from $T$ all vertices with labels at most $j$, and let
$\calC({\bf x})$ be the collection of remaining connected components. Each
$S\in\calC({\bf x})$ is a fringe subtree, rooted at its smallest-label vertex
$r(S)$. The interval in which $S$ must be embedded is
\begin{equation}\label{eq: interval IS definition}
    I_S({\bf x})
    :=
    \left(
        \max\big(\{0\}\cup\{x_i:i\le j,\ \pi(i)<\pi(r(S))\}\big),
        \min\big(\{1\}\cup\{x_i:i\le j,\ \pi(i)>\pi(r(S))\}\big)
    \right).
\end{equation}
Define the potential of $\bf x$
\begin{equation}\label{eq:recursive-value-potential}
    V({\bf x})
    :=
    \sum_{S\in\calC({\bf x})}
    \frac{b(S)}{|I_S({\bf x})|},
\end{equation}
with the convention that the empty sum is zero. In particular,
\[
    V(\varnothing)=b(T),
    \qquad
    V({\bf x})=0\quad\text{for }{\bf x}\in\mathcal H_k.
\]

For $j<k$ and any $z\in(0,1)$, write ${\bf x}\oplus z=(x_1,\ldots,x_j,z)$, and set
$V({\bf x}\oplus z)=+\infty$ when
${\bf x}\oplus z\notin\mathcal H_{j+1}$. We claim that, for every
${\bf x}\in\mathcal H_j$,
\begin{equation}\label{eq:recursive-value-history-certificate}
    \int_0^1
    \big(V({\bf x})-V({\bf x}\oplus z)\big)_+\,dz
    =1.
\end{equation}
Indeed, the next label $j+1$ is the root of one component
$S\in\calC({\bf x})$. Write $I_S({\bf x})=(\ell,u)$. For
$z\in(\ell,u)$, selecting $z$ replaces $S$ by its left and right subtrees,
in the intervals $(\ell,z)$ and $(z,u)$, respectively. Hence
\[
    V({\bf x}\oplus z)
    =
    V({\bf x})
    -\frac{b(S)}{u-\ell}
    +\frac{b(S_L)}{z-\ell}
    +\frac{b(S_R)}{u-z},
\]
where terms corresponding to empty subtrees should be omitted. For
$z\notin(\ell,u)$, the positive part in
\eqref{eq:recursive-value-history-certificate} is zero. Therefore, after the
change of variables $y=(z-\ell)/(u-\ell)$,
\begin{align*}
    \int_0^1
    \big(V({\bf x})-V({\bf x}\oplus z)\big)_+\,dz
    &=
    \int_0^1
    \big(
        b(S)-\phi_{b(S_L),b(S_R)}(y)
    \big)_+\,dy
    =1,
\end{align*}
where the last equality is \eqref{eq:recursive-value-calibration}.

Write $\alg=(\tau_1,\ldots,\tau_k)$, where $\tau_j$ is the time of embedding $\pi(j)$, and set $\tau_0:=0$. For $t\ge0$, define
\[
    J_t:=\max\{0\le j\le k:\tau_j\le t\},
    \qquad
    {\bf X}_t:=(X_{\tau_1},\ldots,X_{\tau_{J_t}}),
\]
with ${\bf X}_t=\varnothing$ when $J_t=0$. Also set
\[
    W_t:=\min\{V({\bf X}_t),b(T)\},
    \qquad
    M_t:=t+W_t.
\]
We claim that $(M_t)_{t\ge0}$ is a submartingale with respect to
$(\calF_t)_{t\ge0}$, where $\calF_0:=\sigma(X_0)$. Fix $t$ and condition on
$\calF_t$. If $J_t=k$, then $W_t=W_{t+1}=0$. Otherwise, let $j:=J_t$ and write
${\bf X}_t={\bf x}\in\mathcal H_j$ and $v:=V({\bf x})$. Conditional on
$\calF_t$, the decision whether to accept $X_{t+1}$ is a measurable function
of $X_{t+1}$. If the value $z$ is rejected, the potential remains $v$; if it
is accepted, the new potential is $V({\bf x}\oplus z)$. Since, for every
$w\in[0,\infty]$,
\[
    \min\{v,b(T)\}-\min\{w,b(T)\}
    \le
    (v-w)_+,
\]
we obtain from \eqref{eq:recursive-value-history-certificate} that
\[
    \E{W_t-W_{t+1}\mid\calF_t}
    \le
    \int_0^1
    \big(V({\bf x})-V({\bf x}\oplus z)\big)_+\,dz
    =1.
\]
Thus $\E{M_{t+1}\mid\calF_t}\ge M_t$, as claimed. In particular,
$\E{W_n}\ge b(T)-n$. On the event $\{T_\alg\le n\}$, the embedding is complete
by time $n$, so $W_n=0$. Since $0\le W_n\le b(T)$,
\[
    b(T)-n
    \le
    \E{W_n}
    \le
    b(T)\,\PP[\stream]{T_\alg>n},
\]
which proves \eqref{eq:finite-horizon-lower-tail}.

Finally, we may assume $T_{\alg}$ is integrable, else
$\E{T_\alg}\ge b(T)$ holds trivially. Moreover, $0\le W_t\le b(T)$ holds almost
surely by construction. So, applying optional stopping gives
\[
    b(T)=M_0
    \le
    \E{M_{T_\alg}}
    =
    \E{T_\alg}
    +
    \E{W_{T_\alg}}.
\]
As $W_{T_\alg}=0$ almost surely, rearranging yields
\[
    \E{T_\alg}\ge b(T)=b(\tree_\pi).
\]
Since $\alg\in\mathrm{OA}(\pi)$ was arbitrary, the result follows.
\end{proof}

\begin{lemma}[Stopping-time recurrence, upper bound]
\label{lemma:fix-rec-UB}
Let $\pi$ be a non-empty permutation of $[k]$. Let $\tau$ be any stopping time defined on the collection of variables $X_0,X_1,\ldots$ with $\EE\,\tau<\infty$, and let
\[
    \alg_L=(\tau_j^L)_{j=1}^{\pi(1)-1},
    \qquad
    \alg_R=(\tau_j^R)_{j=1}^{k-\pi(1)}
\]
be randomized online embedding algorithms for $L(\pi)$ and $R(\pi)$, respectively, with finite expected runtime. There is a randomized online embedding algorithm $\alg=(\tau_j)_{j=1}^k$ for $\pi$ such that
\[
    \EE\tau_k
    =
    \EE\,\tau
    +
    \EE\,\Big(\frac{1}{X_\tau}\Big)\EE\,T_{\alg_L}
    +
    \EE\,\Big(\frac{1}{1-X_\tau}\Big)\EE\,T_{\alg_R}.
\]
If one of $L(\pi)$ and $R(\pi)$ is empty, the corresponding algorithm and term are omitted. In particular, for every subtree $S$ of $\tree_\pi$, there is a recursive interval algorithm $\alg_S^*$ with
\[
    \E{T_{\alg_S^*}}=b(S).
\]
Consequently, $\beta(\pi)\le b(\tree_\pi)$ for every permutation $\pi$.
\end{lemma}

\begin{proof}
We may assume that the extra randomness $X_0$ is a countable collection of iid variables, so we can write
$
    X_0=Y_0\sqcup Y_L\sqcup Y_R,
$
where $\{Y_0,Y_L,Y_R\}$ are independent and distributed the same as $X_0$. We define
$
    \tau_1=\tau(Y_0,X_1,X_2,\ldots)
$
and will inductively define $\tau_2,\ldots,\tau_k$. We work with the random sets
\[
    \Omega_L=\{i>\tau_1:X_i<X_{\tau_1}\},
    \qquad
    \Omega_R=\{i>\tau_1:X_i>X_{\tau_1}\},
\]
noting that $[\tau_1]\cup\Omega_L\cup\Omega_R=\NN$ with probability one.

Suppose that $\tau_1,\ldots,\tau_i$ have been defined for some $i<k$, and assume that $\pi(i+1)<\pi(1)$; if not, the discussion is the same with $L$ replaced by $R$ and $X_{\tau_1}$ replaced by $1-X_{\tau_1}$. Moreover, suppose that up to this point we have constructed two independent lists $\cS_i^L$ and $\cS_i^R$ which are of finite length almost surely; to start,
we let
\[
    \cS_1^L=(Y_L),
    \qquad
    \cS_1^R=(Y_R).
\]
Let $m$ be the index in $L(\pi)$ corresponding to $\pi(i+1)$, and let $g_m$ be the sequence of Borel decision functions defining $\tau_m^L$, as in \cref{sec:ov-notation}. Letting `$[\cdot,\cdot]$' denote concatenation of two sequences, we define
\[
    \tau_{i+1}
    =
    \min\left\{
        t:
        g_m\left(
            \left[
                \cS_i^L,
                (X_j/X_{\tau_1})_
                {j\in\Omega_L\cap[\tau_i+1,t]}
            \right]
        \right)=1
    \right\},
\]
and we set
\[
    \cS_{i+1}^L
    =
    \left[
        \cS_i^L,
        (X_j/X_{\tau_1})_
        {j\in\Omega_L\cap[\tau_i+1,\tau_{i+1}]}
    \right],
    \qquad
    \cS_{i+1}^R=\cS_i^R.
\]
Finally, once $\tau_k$, $\cS_k^L$, and $\cS_k^R$ have been defined, we let
\[
    \cS^L
    =
    \left[
        \cS_k^L,
        (X_j/X_{\tau_1})_
        {j\in\Omega_L\cap[\tau_k+1,\infty)}
    \right]
\]
and
\[
    \cS^R
    =
    \left[
        \cS_k^R,
        \left(
            \frac{X_j-X_{\tau_1}}{1-X_{\tau_1}}
        \right)_
        {j\in\Omega_R\cap[\tau_k+1,\infty)}
    \right].
\]
It is clear by construction that, for each $i\in[k]$, $\tau_i$ is a stopping time for the sequence $X_0,X_1,\ldots$, so it remains to understand
$\EE\tau_k$. Set $\tau_0^L=\tau_0^R:=0$. Since
\[
    \tau_k=\tau_1+\sum_{i=1}^{k-1}(\tau_{i+1}-\tau_i),
\]
it suffices by linearity of expectation to understand $\EE\tau_1$ and
$\EE(\tau_{i+1}-\tau_i)$. By construction, $\EE\tau_1=\EE\tau$. Moreover,
fixing $i$ as above and recalling that $m$ is the index in $L(\pi)$
corresponding to $\pi(i+1)$, we have
\[
    \tau_{i+1}-\tau_i
    =
    |\Omega_L\cap[\tau_i+1,\tau_{i+1}]|
    +
    |\Omega_R\cap[\tau_i+1,\tau_{i+1}]|
\]
and
\[
    |\Omega_L\cap[\tau_i+1,\tau_{i+1}]|
    =
    \tau_m^L-\tau_{m-1}^L.
\]
Let $Z_1,Z_2,\ldots$ be conditionally independent geometric variables with
random parameter $X_{\tau_1}$. Then
\[
    |\Omega_R\cap[\tau_i+1,\tau_{i+1}]|
    =
    \sum_{j=1}^{\tau_m^L-\tau_{m-1}^L}(Z_j-1).
\]
Thus, Wald's identity, together with the independence of
$\tau_m^L-\tau_{m-1}^L$ from $X_{\tau_1}$, gives
\[
    \EE|\Omega_R\cap[\tau_i+1,\tau_{i+1}]|
    =
    \EE[\tau_m^L-\tau_{m-1}^L]\EE[Z_1-1].
\]
Consequently,
\begin{align*}
    \EE[\tau_{i+1}-\tau_i]
    =
    \EE[\tau_m^L-\tau_{m-1}^L]
    +
    \EE|\Omega_R\cap[\tau_i+1,\tau_{i+1}]| &=
    \EE[\tau_m^L-\tau_{m-1}^L]\EE Z_1 \\
    &=
    \EE\left(\frac{1}{X_{\tau_1}}\right)
    \EE[\tau_m^L-\tau_{m-1}^L].
\end{align*}
It follows that
\[
    \sum_{i:\,\pi(i+1)<\pi(1)}
    \E{\tau_{i+1}-\tau_i}
    =
    \EE\left(\frac{1}{X_{\tau_1}}\right)
    \E{T_{\alg_L}}.
\]
Similarly,
\[
    \sum_{i:\,\pi(i+1)>\pi(1)}
    \EE[\tau_{i+1}-\tau_i]
    =
    \EE\left(\frac{1}{1-X_{\tau_1}}\right)
    \EE\,T_{\alg_R}.
\]
Since $(\tau_1,X_{\tau_1})$ has the same distribution as
$(\tau,X_\tau)$, we conclude that
\[
    \EE\tau_k
    =
    \EE\tau
    +
    \EE\left(\frac{1}{X_\tau}\right)\EE\,T_{\alg_L}
    +
    \EE\left(\frac{1}{1-X_\tau}\right)\EE\,T_{\alg_R}.
\]

It remains to prove the assertion involving $b$. We proceed by induction on
the size of a permutation $\sigma$, writing $S:=\tree_\sigma$. The empty
permutation is trivial. Suppose that $\sigma\neq\emptyset$, and let
\[
    p:=b(S_L),
    \qquad
    q:=b(S_R).
\]
Let $(a_*,b_*)$ attain the minimum defining $G(p,q)$, whose existence follows
from Proposition~\ref{prop:F-root}, and let
\[
    \tau_*:=\min\{t\ge1:X_t\in[a_*,b_*]\}.
\]
By the induction hypothesis, there are recursive interval algorithms for the
left and right subpermutations with expected runtimes $p$ and $q$,
respectively. Applying the construction above with these two algorithms and
the stopping time $\tau_*$ gives a recursive interval algorithm $\alg_S^*$ satisfying
\begin{align*}
    \E{T_{\alg_S^*}}
    &=
    \frac{1}{b_*-a_*}
    +
    \frac{1}{b_*-a_*}
    \int_{a_*}^{b_*}\phi_{p,q}(x)\,dx =
    G_{p,q}(a_*,b_*)
    = G(p,q) = b(S).
\end{align*}
Here and above, the corresponding term is omitted when a child is empty. This completes the induction and proves $\beta(\pi)\le b(\tree_\pi)$.
\end{proof}

\begin{proof}[Proof of Theorem~\ref{thm:recursion-formula}]
For non-empty $\pi$, Proposition~\ref{prop:recursive-value-lower-bound} and
Lemma~\ref{lemma:fix-rec-UB} give
\[
    \beta(\pi)=b(\tree_\pi).
\]
The same identity is immediate for $\pi=\emptyset$. Therefore, if
$\pi\neq\emptyset$,
\[
    \beta(\pi)
    =
    b(\tree_\pi)
    =
    G\big(b(\tree_L),b(\tree_R)\big)
    =
    G\big(\beta(L(\pi)),\beta(R(\pi))\big),
\]
as required.
\end{proof}

The finite-horizon estimate \eqref{eq:finite-horizon-lower-tail}, together
with $\beta(\pi)=b(\tree_\pi)$, also proves \eqref{eq:fixed-prob-LB}. Indeed,
with $n:=\lfloor(1-\delta)\beta(\pi)\rfloor$,
\[
    \PP[\stream]{T_\alg\le(1-\delta)\beta(\pi)}
    =
    \PP[\stream]{T_\alg\le n}
    \le
    \frac{n}{\beta(\pi)}
    \le
    1-\delta.
\]

\subsection{The optimal embedding algorithm}
The algorithm below is a pseudocode description of the optimal online embedding algorithm, henceforth called the \emph{canonical algorithm}. The algorithm uses subroutines $a_*(T)$ and $b_*(T)$ that take in a binary tree and output $a$ and $b$ corresponding to the interval $(a,b)$ achieving the optimum in \eqref{eq:rec fixed} for the permutation corresponding to $T$. A binary tree $T$ is a list $(v_1,\dots, v_k)$ where each $v \in T$ has data $v.\ell$, $v.u$ and $v.y$ (which will be given values as the algorithm runs), corresponding respectively to the  lower and upper bounds of the acceptance interval and the embedded value of $v$, as well as pointers $v.L$ and $v.R$ to its left and right children.
\begin{algorithm}[h]
\caption{Canonical algorithm for embedding general permutations}
\label{alg:canonical-general2}
\begin{algorithmic}[1]

\State Initialize $t \gets 0$, $(v_1,\dots,v_k) \gets \tree_\pi$, and $(v_1.\ell,v_1.u) \gets (0,1)$
\For{$i = 1:k$}
    \State $\tree_i \gets$ fringe sub-tree rooted at $i$
    \State $I \gets  \big(
            v_i.\ell+(v_i.u-v_i.\ell)a_*(\tree_i),\,
            v_i.\ell+(v_i.u-v_i.\ell)b_*(\tree_i)
        \big)$ 
    \Repeat
        \State \(t\gets t+1\), and reveal \(X_t\).
    \Until{\(X_t\in I\)}
    \State Set \(\tau_i\gets t\) and \(v_i.y \gets X_t\). \\
           
    \State $L,R \gets$ left and right sub-trees rooted at $i$ 
    \State $(v_L,v_R) \gets $ references to left and right child of $v_i$
    \State $(v_L.\ell, v_L.u) \gets (v_{i}.\ell,\,v_i.y)$
    \State $(v_R.\ell, v_R.u) \gets (v_i.y,\, v_i.r)$
\EndFor
\end{algorithmic}
\end{algorithm}

\section{Probability bounds}\label{sec:prob}

\subsection{Fixed patterns}
The lower-tail estimate \eqref{eq:fixed-prob-LB} for arbitrary online algorithms was proved at the end of the
preceding section. We now turn to the concentration estimate
\eqref{eq:fixed-prob-UB} for the canonical algorithm.
For a finite rooted binary tree $\tree$, let $\alg^*_\tree$ denote the
canonical embedding algorithm for $\tree$, set
$T^*_\tree:=T_{\alg^*_\tree}$, and define
\[
    v(\tree):=\pVar{T^*_\tree}.
\]

\begin{proposition}[Variance bound for the canonical algorithm]\label{prop:variance-canonical}
There is a finite universal constant $C$ such that the following holds. Fix any $k <\infty$, let $S$ be a fringe subtree of $\tree_\pi$ for some $k$-permutation $\pi$, and let $\alg_S^*$ be the canonical recursive algorithm for embedding $S$. Then,
\[
\EE T_{\alg_S^*}=\beta(S)
\qquad\text{and}\qquad
\pVar{T_{\alg_S^*}}\le C\,\beta(S)^{3/2}.
\]
Consequently,
\[
\PP[\stream]{\ba{T_{\alg_S^*}-\beta(S)} > y\,\beta(S)^{3/4}} \le \frac{C}{y^2},
\qquad y\ge 1.
\]
\end{proposition}

We will need a second-moment version of \eqref{eq:rec fixed}.

\begin{lemma}\label{lemma:fixed-rec-var}
Let \(\tree\) be a non-empty finite rooted binary tree, and let \(L\) and
\(R\) denote its left and right subtrees. Let \(\tau := \tau_1\) denote the stopping time for the root of $\tree$ under the canonical algorithm. Then for almost every $x\in(0,1)$,
\[
    \EE[T^*_\tree \mid X_\tau=x]
    =
    \EE[\tau\mid X_\tau=x]
    + \frac{\beta(L)}{x}
    + \frac{\beta(R)}{1-x},
\]
and
\[
    \pVar{T^*_\tree\mid X_\tau=x} =
    \pVar{\tau\mid X_\tau=x} 
    + \frac{v(L)}{x^2}
    + \frac{v(R)}{(1-x)^2}
    + \frac{1-x}{x^2}\,\beta(L)
    + \frac{x}{(1-x)^2}\,\beta(R).
\]
\end{lemma}

\begin{proof}
We condition throughout on the event \(\{X_{\tau}=x\}\), in the sense of regular conditional distributions. Recall that after embedding $\pi(1)$, the canonical algorithm recursively embeds the left and right subtrees again via the canonical algorithm. Write
\[
    T^*_\tree=\tau+A_L+A_R,
\]
where \(A_L\) and \(A_R\) are the time spent while
embedding the left and right subtrees after the root. We compute the conditional mean of \(A_L\); the computation for \(A_R\) is symmetric. If \(L=\emptyset\),
then \(A_L=0\). Otherwise, let \(G_1^{(x)},G_2^{(x)},\ldots\) be iid geometric
random variables with parameter \(x\), independent of
\(T^*_L\). The construction of the canonical algorithm gives the distributional
identity
\begin{equation}\label{eq:timechange}
    A_L \stackrel{d}{=} \sum_{j=1}^{T^*_L} G_j^{(x)}.
\end{equation}
Indeed, ``lower stream'' observations used by the recursive algorithm (that is, arrivals $X_t$ with $t > \tau$ and $X_t < x$)  appear after a geometric number of observations with
success parameter \(x\); the claimed identity then follows from Bernoulli thinning and the strong Markov property. Hence
\[
    \EE[A_L\mid X_\tau=x]=\frac{\beta(L)}{x}\,.
\]
In particular, applying the law of total variance to \eqref{eq:timechange},
\begin{align*}
    \pVar{A_L\mid X_\tau=x}
    &= \EE[\,\pVar{A_L \mid T_L^*
    , \,X_\tau = x} \mid X_\tau = x] + \pVar{\,\EE[A_L \mid T_L^*,\, X_\tau = x] \mid X_\tau = x} \\
    &=
    \EE[T^*_L]\,\frac{1-x}{x^2}
    +
    \pVar{T^*_L}\,\frac1{x^2}\\
    &=
    \frac{1-x}{x^2}\beta(L)+\frac{v(L)}{x^2}.
\end{align*}
Similarly,
\[
    \EE[A_R\mid X_\tau=x]=\frac{\beta(R)}{1-x}, \quad 
    \pVar{A_R\mid X_\tau=x}
    =
    \frac{x}{(1-x)^2}\beta(R)+\frac{v(R)}{(1-x)^2}.
\]
Conditional on \(\{X_\tau=x\}\), the variables $\tau$, $A_L$, and $A_R$ are mutually independent by construction of the canonical algorithm. Hence,
\[
    \EE[T^*_\tree\mid X_\tau=x]
    =
    \EE[\tau\mid X_\tau=x]
    +\frac{\beta(L)}{x}
    +\frac{\beta(R)}{1-x},
\]
and
\[
    \pVar{T^*_\tree\mid X_\tau=x}
    =
    \pVar{\tau\mid X_\tau=x}
    +\frac{v(L)}{x^2}
    +\frac{v(R)}{(1-x)^2}
    +\frac{1-x}{x^2}\beta(L)
    +\frac{x}{(1-x)^2}\beta(R).
\]
\end{proof}

Next, we record basic estimates on function $G$ from \eqref{eq:def-G1}. 

\begin{lemma}\label{lem:G-boundary-lower}
Let $p,q \ge 0$. Then
    \begin{equation}\label{eq:G1}
        G(p,q) \ge (\sqrt{p} +\sqrt{q})^2,
    \end{equation}
and
    \begin{equation}\label{eq:G2}
    G(0,q) \ge (\sqrt{q} + 1/2)^2.
    \end{equation}
The symmetric bound for $G(p,0)$ holds as well.   
\end{lemma}

\begin{proof}
The Cauchy-Schwarz inequality implies $p/x + q/(1-x) \ge (\sqrt{p} + \sqrt{q})^2$ for all $x \in (0,1)$. Applying this to the definition of $G$ gives the first claim. For the second claim, fix $0<a<b<1$ and set
\[
    d:=\frac{b-a}{1-a}\in(0,1).
\]
Then $b-a=(1-a)d\le d$ and
\[
    \log\frac{1-a}{1-b}=-\log(1-d).
\]
Thus it suffices to prove
\[
    \frac{1+q\pa{-\log(1-d)}}{d}
    \ge
    \pa{\sqrt q+\frac12}^2 .
\]
Using $-\log(1-d)\ge d+d^2/2$, we get
\begin{align*}
    1+q\pa{-\log(1-d)}
    -d\pa{\sqrt q+\frac12}^2
    &\ge
    1+\frac{d^2q}{2}-d\sqrt q-\frac d4  \\
    &=
    \frac{d^2}{2}\pa{\sqrt q-\frac1d}^2
    +\frac12-\frac d4
    \ge 0 .
\end{align*}
Since this holds for every $0<a<b<1$, taking the infimum over $(a,b)$ proves the claim.
\end{proof}

With the above lemmas in hand, we prove the variance bound by induction on the size of the tree.

\begin{proof}[Proof of Proposition~\ref{prop:variance-canonical}] The expectation identity follows from Lemma~\ref{lemma:fix-rec-UB} and Theorem~\ref{thm:recursion-formula}. The empty tree is trivial, so fix a non-empty tree $\tree$, and let $L$ and $R$ denote the left and right subtrees of its root. It remains to prove the variance bound.
It will be convenient to have the shorthand:
\[
B:=\beta(\tree), \qquad B_L:=\beta(L), \qquad B_R:=\beta(R).
\]

We first treat the case $L\neq\emptyset$ and $R\neq\emptyset$. Let $(a,b)$
be an optimizer in \eqref{eq:rec fixed}, set $\Delta:=b-a$, and let
$U:=X_{\tau_1}$, so that $U\sim \mathrm{Unif}(a,b)$. Set
\[
f(x):=\frac{B_L}{x}+\frac{B_R}{1-x},
\qquad
r(x):=B-f(x),
\qquad x\in[a,b].
\]

Thus \(f(x)\) is the conditional expected downstream time to embed, while
\(r(x)=B-f(x)\) is the slack compared to the unconditional value \(B\). Then $(a,b)$ minimizes
\[
    \frac{1+\int_a^b f(x)\,dx}{b-a}.
\]
In this two-child case the optimizer is interior, and the first-order
stationary conditions give
\begin{equation}\label{eq:opt-endpoints}
B=f(a)=f(b).
\end{equation}
Hence $r(a)=r(b)=0$. Since $f$ is convex on $(0,1)$, the function $r$ is concave on $[a,b]$. Moreover, we have from \eqref{eq:rec fixed} that:
\begin{equation}\label{eq:r-area}
\int_a^b r(x)\,dx=(b-a)\pa{B - \frac{1}{b-a}\int_a^b \frac{B_L}{x} + \frac{B_R}{1-x}\,dx} = 1.
\end{equation}
By concavity of $r$, we thus obtain the crude bound:
\begin{equation}\label{eq:r-sup}
\sup_{x\in[a,b]} r(x)\le \frac{2}{\Delta}.
\end{equation}
This inequality gives a deterministic upper bound on the amount that the actual value of $U = X_{\tau_1}$ can affect the downstream expected embedding time. Next, we lower bound $\Delta := b-a$. From \eqref{eq:opt-endpoints}, we have
\begin{equation}\label{eq:ab-identities}
ab=\frac{B_L}{B},
\qquad
(1-a)(1-b)=\frac{B_R}{B}.
\end{equation}
Further, solving $f'(x)=0$ yields:
\[
\min_{x\in(0,1)} f(x)=(\sqrt{B_L}+\sqrt{B_R})^2.
\]
We extract two important consequences from this identity. First, as an aside, we record for later usage:
\begin{equation}\label{eq:B f comparison}
    B \stackrel{\eqref{eq:opt-endpoints}}{=} f(a)
    \ge \min_{x \in (0,1)} f(x)
    = (\sqrt{B_L}+\sqrt{B_R})^2.
\end{equation}

Second, combining the identity with \eqref{eq:r-area} gives
\begin{equation}\label{eq:integral-bound}
    1\stackrel{\eqref{eq:r-area}}{=} \int_a^b r(x)dx \le \Delta\Big(B-(\sqrt{B_L}+\sqrt{B_R})^2\Big).
\end{equation}
On the other hand, the explicit formula for the two roots of $f(x)=B$ yields
\[
\Delta
=
\frac{\sqrt{\big(B-(\sqrt{B_L}+\sqrt{B_R})^2\big)\big(B-(\sqrt{B_L}-\sqrt{B_R})^2\big)}}{B}.
\]
Using the inequality $(\sqrt{B_L}+\sqrt{B_R})^2 \geq (\sqrt{B_L}-\sqrt{B_R})^2$, valid because $B_L,B_R \ge 0$, yields
\begin{equation}\label{eq:B_Delta-bound}
   B\Delta \ge B-(\sqrt{B_L}+\sqrt{B_R})^2. 
\end{equation}
Thus, we obtain the desired lower bound on $\Delta$:
\begin{equation}\label{eq:delta-bound}
\Delta^{-2}\le B.
\end{equation}

We are ready to apply \cref{lemma:fixed-rec-var}. Note that $\tau := \tau_1$ is geometric with parameter $\Delta$ and is independent of $U$. Recall that the variance of a geometric with parameter $p$ is $(1-p)/p^2$. Conditioning on $U$ using the law of total variance, and then applying \cref{lemma:fixed-rec-var}, we obtain:
\begin{align}
 v(\tree)
 &= \frac{1-\Delta}{\Delta^2}
 + \EE\br{\frac{v(L)}{U^2}}
 + \EE\br{\frac{v(R)}{(1-U)^2}} \notag\\
 &\quad
 + B_L\,\EE\br{\frac{1-U}{U^2}}
 + B_R\,\EE\br{\frac{U}{(1-U)^2}}
 + \pVar{\frac{B_L}{U}+\frac{B_R}{1-U}}.
 \label{eq:variance-recursion-exact}
\end{align}
It remains to bound each of the terms on the right-hand side of this identity using $B$, $B_L$, and $B_R$ in order to ``close'' the recursion. By \eqref{eq:ab-identities}, 
\[
\EE\br{U^{-2}}=\frac{1}{ab}=\frac{B}{B_L},
\qquad
\EE\br{(1-U)^{-2}}=\frac{1}{(1-a)(1-b)}=\frac{B}{B_R}.
\]
Here, we are assuming that $L$ and $R$ are non-empty in the first and second equations, respectively, so that $B_L$ and $B_R$ are non-zero. (We will not need these bounds in the case that $L$ or $R$ is empty). Thus, noting that $v(L)$ and $v(R)$ are non-random scalars, linearity of expectation yields
\[
\EE\br{\frac{v(L)}{U^2}}=\frac{B}{B_L}v(L).
\qquad
\EE\br{\frac{v(R)}{(1-U)^2}}=\frac{B}{B_R}v(R).
\]
Similarly, using the trivial bound that $U \in [0,1]$, we have:
\[
B_L\,\EE\br{\frac{1-U}{U^2}}\le B_L\,\EE\br{\frac{1}{U^2}}= B,
\qquad
B_R\,\EE\br{\frac{U}{(1-U)^2}} \le B_R\,\EE\br{\frac{1}{(1-U)^2}} = B.
\]
Finally, since
\[
\frac{B_L}{U}+\frac{B_R}{1-U}=B-r(U),
\]
we have, by \eqref{eq:r-area}, \eqref{eq:r-sup}, and \eqref{eq:delta-bound},
\[
\pVar{\frac{B_L}{U}+\frac{B_R}{1-U}}
= \pVar{r(U)}
\le \EE[r(U)^2]
\le \frac{2}{\Delta}\cdot \frac{1}{\Delta}
\le 2B.
\]
Substituting into \eqref{eq:variance-recursion-exact} yields
\begin{equation}\label{eq:variance-recursion-final}
 v(\tree)
 \le C_0 B
 + \frac{B}{B_L}v(L)\,\ind_{\{L\neq\varnothing\}}
 + \frac{B}{B_R}v(R)\,\ind_{\{R\neq\varnothing\}}.
\end{equation}

Consider the following hypothesis, which we aim to establish for every finite tree $\tree$: 
\begin{equation}\label{eq:var induction goal}
    v(\tree)\le 2C_0\,\beta(\tree)^{3/2} - C_0 \, \beta(\tree).
\end{equation}

We now prove \eqref{eq:var induction goal} by induction on the number of vertices of \(\tree\). The one-vertex tree has \(\beta(\tree)=1\) and \(v(\tree)=0\). Now let
\(|\tree|\ge2\), and assume the claim holds for all proper subtrees of
\(\tree\). Then it applies to the non-empty members of \(\{L,R\}\).

\textbf{Case 1:} ($L$ and $R$ are non-empty). Applying \eqref{eq:variance-recursion-final}, followed by the inductive hypothesis, and then finally \eqref{eq:B f comparison}, we have:
\begin{align*}
    v(\tree) &\le C_0 B
 + \frac{B}{B_L}v(L)\,\ind_{\{L\neq\varnothing\}}
 + \frac{B}{B_R}v(R)\,\ind_{\{R\neq\varnothing\}} \\
 &\le C_0B +  B(2C_0\sqrt{B_L} - C_0) + B(2C_0\sqrt{B_R} - C_0) \\
 &\le 2C_0 B^{3/2} - C_0 B \,.
\end{align*}
The induction is complete in this case. \\

\textbf{Case 2:} ($L$ or $R$ is empty). If $L = R = \emptyset$, we are in the base-case of $B = 1$, so we may assume that exactly one of $L$ and $R$ is empty. Without loss of generality, say $R = \emptyset$. 
By \eqref{eq:G2},
\begin{equation}\label{eq:B f comparison 1child}
    B \ge  \left(\sqrt{B_L} + \frac{1}{2}\right)^2. 
\end{equation}
Applying \eqref{eq:variance-recursion-final} and the inductive hypothesis, followed by \eqref{eq:B f comparison 1child},
\begin{align*}
    v(\tree)
 &\le C_0B +  B\left(2C_0\sqrt{B_L} - C_0 \right) \le C_0B + B\left(2C_0\left(\sqrt{B} - \frac{1}{2}\right) - C_0\right) = 2C_0B^{3/2} - C_0B \,.
\end{align*}
This completes the induction and establishes \eqref{eq:var induction goal}. The bound \eqref{eq:fixed-prob-UB} follows immediately from Chebyshev's inequality.

\end{proof}

\subsection{Proof of concentration claims in \cref{thm:main-random}}

Our goal is to prove \eqref{eq:goal-beta-concentration}. The strategy consists of four steps. First, we will show that $\tree_\pi$ for $\pi \sim \mu_k$ has the distribution of a random binary search tree (BST), a well-studied model that is defined shortly. Second, we will show that, for some bounded function $\eta$ which maps trees to $[0,1]$, we can write $\sqrt{\beta(\tree)}$ as a sum of $\eta$ over subtrees:
\[
    \sqrt{\beta(\tree)} =  \sum_{v \in \tree} \eta(\tree_v)\,.
\]
The main mathematical content here is showing that $\eta$ is bounded. 
Third, we will utilize one of the numerous readily available concentration results for such functionals on random BSTs, yielding variance bounds as well as a CLT. Unfortunately, a general drawback of the utilized machinery is that we only obtain a CLT with a finite non-negative variance parameter $\sigma^2 \ge 0$. Consequently, as a fourth and final step, we prove a fluctuation lower bound, implying $\sigma^2 > 0$ and hence pin down the exact asymptotic order of the fluctuations of $\beta(\pi)$ as $k^{3/2}$. We begin with the standard definition of a random BST:

\begin{definition}[BST]\label{def:BST}
Let $\pi \in \Pi_k$ be a permutation. The binary search tree 
$\mathcal{T}_{\mathrm{BST},\pi}$ is constructed by inserting nodes labeled $1,\dots,k$ (containing keys
$\pi(1),\pi(2),\dots,\pi(k)$, respectively) sequentially into an initially empty binary search tree. The first node, labeled 1, is inserted as the root. Each remaining node $i$ is sequentially inserted as follows. Start at the root; descend to the left if the key $\pi(i)$ is smaller 
than the key of the current node, and descend to the right otherwise. Repeat this descent until an empty position is reached, and insert the current node at this empty position. If $\pi$ was drawn uniformly at random from $\Pi_k$, then we further say that $\mathcal{T}_{\mathrm{BST},\pi}$ is a \textit{random} BST on $[k]$.
\end{definition}

\begin{proposition}[Equivalence of tree constructions]\label{prop:tree is rBST}
    Given a permutation $\pi$, the labeled trees $\tree_\pi$ and  $\mathcal{T}_{\mathrm{BST},\pi}$ are equal. In particular, if $\pi \sim \mu_k$, then $\tree_\pi$ is distributed as a random BST on $[k]$. 
\end{proposition}

\begin{proof}
We proceed by induction on $k = |\pi|$. For $k\le 1$, the claim is immediate. Now let $k\ge 2$, and assume the claim holds for all smaller permutations. Let $\pi\in \Pi_k$. Both $\tree_\pi$ and $\mathcal{T}_{\mathrm{BST},\pi}$ have root node with label 1. It is also easily checked that every node with label $i > 1$ satisfying $\pi(i) < \pi(1)$ is, by construction, embedded somewhere into the left subtree of the root. Similarly, every $i > 1$ with $\pi(i) > \pi(1)$ is embedded into the right subtree of the root. Hence the list of node labels in $\tree_L$ and $\tree_R$ are the same for all three methods of constructing the tree $\tree_\pi$. The result follows from the induction hypothesis. 
\end{proof}

With this observation in hand, we aim to express $\sqrt{\beta}$ in terms of a bounded functional on trees, so as to apply existing concentration results.

\begin{lemma}[Bounded tree functional]\label{lem:eta bounded}
For every non-empty finite rooted binary tree $\tree$ with left and right subtrees $\ltree$ and $\rtree$,
\[
    0 \le \sqrt{\beta(\tree)}-\sqrt{\beta(\ltree)}-\sqrt{\beta(\rtree)} \le 1.
\]
In particular, if we define
\[
    \eta(\tree)
    :=
    \sqrt{\beta(\tree)}-\sqrt{\beta(\ltree)}-\sqrt{\beta(\rtree)},
\]
then $0\le \eta(\tree)\le 1$.
\end{lemma}

\begin{proof} Denote
\[
    B:=\beta(\tree),\qquad
    B_L:=\beta(\ltree),\qquad
    B_R:=\beta(\rtree),
\]
and
\[
    p:=\sqrt{B_L},\qquad q:=\sqrt{B_R},\qquad s:=p+q.
\]
The lower bound follows immediately from the first assertion of Lemma~\ref{lem:G-boundary-lower}
and the recursive formula for $\beta(\cdot)$. 
For the upper bound, fix $\eps>0$ and define
\[
    a_\eps:=\frac{p+\eps}{s+1+2\eps},
    \qquad
    b_\eps:=\frac{p+1+\eps}{s+1+2\eps}.
\]
The purpose of the regularization term $\eps$ is to enforce the strict inequalities $0<a_\eps<b_\eps<1$. We have
\[
    b_\eps-a_\eps=\frac{1}{s+1+2\eps}.
\]
Moreover, for every $x\in[a_\eps,b_\eps]$ we have
\[
    x\ge a_\eps=\frac{p+\eps}{s+1+2\eps},
    \qquad
    1-x\ge 1-b_\eps=\frac{q+\eps}{s+1+2\eps},
\]
so
\[
    \frac{B_L}{x}
    \le p^2\frac{s+1+2\eps}{p+\eps}
    \le p(s+1+2\eps),
\]
and similarly
\[
    \frac{B_R}{1-x}
    \le q(s+1+2\eps).
\]
Therefore, for every $x\in[a_\eps,b_\eps]$,
\[
    \frac{B_L}{x}+\frac{B_R}{1-x}
    \le s(s+1+2\eps).
\]
Evaluating the recurrence at $(a_\eps,b_\eps)$ gives
\begin{align*}
    B
    \le
    \frac{1}{b_\eps-a_\eps}
    +
    \frac{1}{b_\eps-a_\eps}
    \int_{a_\eps}^{b_\eps}
        \frac{B_L}{x}+\frac{B_R}{1-x}
    \,dx 
    &\le (s+1+2\eps)+s(s+1+2\eps)
    \\ &=(s+1)(s+1+2\eps).
\end{align*}
Letting $\eps\downarrow 0$ yields
$B\le (s+1)^2$. Hence $\sqrt{B}\le \sqrt{B_L} + \sqrt{B_R}+1$, as claimed.
\end{proof}

\begin{corollary}[Representation of $\sqrt{\beta}$ as linear functional]\label{cor:additive}
For every finite rooted binary tree $\tree$,
\[
    \sqrt{\beta(\tree)}
    =
    \sum_{v\in V(\tree)} \eta(\tree_v),
\]
where $\tree_v$ denotes the subtree rooted at $v$. 
\end{corollary}

\begin{proof}
Iteratively apply the identity
\[
    \sqrt{\beta(\tree)} = \sqrt{\beta(\ltree)}+\sqrt{\beta(\rtree)}+\eta(\tree)
\]
to its own right-hand side. All of the $\beta$ terms on the right-hand side expand using the
same identity, leaving exactly the claimed sum of $\eta$ terms. 
\end{proof}

\begin{remark}
    Using the bound $\eta\in[0,1]$ recovers the naive upper bound
    \[
        \beta(\tree)\le |\tree|^2,
    \]
    corresponding to the trivial algorithm of embedding each \(j\in[k]\) in
the window \([(\pi(j)-1)/k,\pi(j)/k]\).
\end{remark}

We now need some results from the literature of random BSTs, due to Holmgren and Janson \cite{tree-CLT}. In the terminology of random BSTs, a \textit{fringe tree} is simply a maximal rooted subtree, and a toll function is a functional on fringe trees. Finally, a function on a BST which can be written as a toll function summed over all fringe trees, such as $\sqrt{\beta}$ in \cref{cor:additive}, is called an \textit{additive functional}. Here, $\eta$ is a bounded toll function and $\sqrt{\beta(\cdot)}$ is the corresponding additive functional:
\[
    \sqrt{\beta(\tree)} = \sum_{v\in V(\tree)} \eta(\tree_v).
\]

\begin{theorem}[Holmgren and Janson  {\cite[Theorem~1.14]{tree-CLT}}]\label{thm:HJ CLT}
Let $\tree_k$ be a random BST on $k$ nodes, and let $f$ be a function mapping unlabeled fringe trees to $[0,1]$. Define the following $\tree$-measurable random quantity:
\[
    F_k := \sum_{v\in V(\tree_k)} f\pa{(\tree_k)_v}.
\]
Then, there exists a constant $\sigma^2\ge 0$ depending only on $f$ such that $\pVar{F_k}=\sigma^2 k + \ao{k}$. Moreover, 
\begin{equation}\label{eq:CLT}
    \frac{F_k-\EE[F_k]}{\sqrt{k}}\stackrel{d}{\longrightarrow} N\pa{0,\sigma^2}.
\end{equation}
\end{theorem}

We now deduce the main result of this section, save for the variance lower bound $\sigma^2 > 0$. 

\begin{proof}[Proof of \cref{thm:main-random} (with $\sigma^2 \ge 0$)]
Let $\pi \sim\mu_k$, let $\tree=\tree_\pi$, and set
\[
    H_k:=\sqrt{\beta(\tree)},\qquad m_k:=\EE H_k.
\]
The recurrence \eqref{eq:rec fixed} implies that \(\beta(\tree)\), and hence
\(\eta(\tree)\), depends only on the rooted unlabeled shape of \(\tree\).
Together with \cref{lem:eta bounded,cor:additive}, this shows that
\(f:=\eta\) is a bounded toll function on unlabeled fringe trees, taking values
in \([0,1]\). Therefore \cref{thm:HJ CLT} gives, for some \(\sigma_H^2\ge0\),
\[
    \frac{H_k-m_k}{\sqrt{k}}
    \stackrel{d}{\longrightarrow}
    \sigma_H Z,
    \qquad
    \pVar{H_k}=\sigma_H^2 k+o(k),
\]
where \(Z\) is standard normal. Using \(\beta_k^\typ=\EE H_k^2=m_k^2+\pVar{H_k}\) and the existence of
\(c_\typ=\lim_k\beta_k^\typ/k^2\), proved below in
\cref{lem:subsupadd}, we have
\begin{equation}\label{eq:mk-limit}
    \frac{m_k}{k}\longrightarrow \sqrt{c_\typ}.
\end{equation}
Now
\[
\frac{\beta(\pi)-\beta_k^\typ}{k^{3/2}}
=
2\frac{m_k}{k}\frac{H_k-m_k}{\sqrt{k}}
+
\frac{(H_k-m_k)^2-\pVar{H_k}}{k^{3/2}}.
\]
The second term converges in probability to zero, while the first term converges in
distribution to \(2\sqrt{c_\typ}\,\sigma_H Z\). Hence
\[
    \frac{\beta(\pi)-\beta_k^\typ}{k^{3/2}}
    \stackrel{d}{\longrightarrow}
    2\sqrt{c_\typ}\,\sigma_H Z.
\]
This proves the CLT, with limiting variance
\(\sigma^2=4c_\typ\sigma_H^2 \ge 0\). The concentration statement
\eqref{eq:goal-beta-concentration} follows.
\end{proof}

As an interlude, we naturally obtain control over $\sqrt{c_\typ}$, the mean of the above CLT, and hence over $c_\typ$. This result will be used shortly in \cref{sec:scaling} to obtain rigorous estimates on $c_\typ$.

\begin{proposition}[Tight lower-envelope on $c_\typ$]
  \begin{equation}\label{eq:ctyp-LB}
    c_\typ = \sup_{k \in \NN} \Big(\frac{2}{(k+1)(k+2)} \sum_{j=1}^k \sqrt{\beta_j^\typ}\Big)^2 \geq \sup_{k \in \NN} \Big(\frac{2}{(k+1)(k+2)} \sum_{j=1}^k \sqrt{\beta_j^-}\Big)^2.
\end{equation}  
\end{proposition}

\begin{proof}
    We continue to use the notation of the previous proof. Additionally, denote the sequence in the right-hand side of \eqref{eq:ctyp-LB} by
    \begin{equation}\label{eq:ctyp-lower-sequence}
    A_k := \frac{2}{(k+1)(k+2)}\sum_{j=1}^k m_j.
\end{equation}
We claim both that $A_k$
is non-decreasing in \(k\) and that $A_k \to \sqrt{c_\typ}$. Towards establishing monotonicity, we have from
\cref{lem:eta bounded} that every nonempty tree \(T\) satisfies
\[
    \sqrt{\beta(T)}
    \ge
    \sqrt{\beta(T_L)}+\sqrt{\beta(T_R)}.
\]
We apply this to $\tree_k$ and its subtrees $\tree_L$ and $\tree_R$. Note that $|\tree_L|$ is uniform on
\(\{0,\dots,k-1\}\). Hence, on the event $\{|\tree_L| = i\}$, it holds that $\tree_L$ and $\tree_R$ are conditionally distributed as random BSTs on \(i\) and \(k-1-i\)
vertices, respectively. Taking expectations gives
\begin{equation}\label{eq:mk-superaverage}
    m_k
    \ge
    \frac1k\sum_{i=0}^{k-1}
        \bigl(m_i+m_{k-1-i}\bigr)
    =
    \frac{2}{k}\sum_{j=1}^{k-1}m_j.
\end{equation}
The claimed monotonicity of $A_k$ follows:
\begin{align*}
    A_k - A_{k-1} =\frac{2}{(k+1)(k+2)}\sum_{j=1}^k m_j - \frac{2}{k(k+1)}\sum_{j=1}^{k-1}m_j  = \frac{2}{(k+1)(k+2)} \left(m_k-\frac{2}{k}\sum_{j=1}^{k-1}m_j\right) \ge0.
\end{align*}
Next, we establish the claim that $A_k \to \sqrt{c_\typ}$. We compute
\begin{align*}
\lim_{k\to\infty}
\frac{2}{(k+1)(k+2)}\sum_{j=1}^k m_j = \lim_{k\to\infty}
\frac{m_k}{k+1} \stackrel{\eqref{eq:mk-limit}}{=} \sqrt{c_\typ}.
\end{align*}
The first equality uses the Stolz--Ces\`aro Theorem to evaluate the sum. Hence, for every \(k\ge1\), non-negativity of $m_j$ implies
\[
    \sqrt{c_\typ}
    \ge
    \frac{2}{(k+1)(k+2)}\sum_{j=1}^k m_j.
\]
Finally, $\beta(\tree_j)\ge\beta_j^-$ almost surely for every $j$, by definition of $\beta_j^-$. Hence, 
\[
    m_j=\EE\sqrt{\beta(\tree_j)}
    \ge\sqrt{\beta_j^-}.
\]
Substituting this into the preceding display completes the result. 
\end{proof}

\subsection{Lower bound on fluctuations}

Finally, we prove $\sigma_H^2$, and hence $\sigma^2$, is strictly positive. We need the following estimate:
\begin{lemma}[Lower slope of the square-root recursion]
\label{lem:sqrt-rec-lower-slope}
For \(x,y\ge0\), set
\[
    \Psi(x,y)
    := \sqrt{G(x^2,y^2)}.
\]
Then, for all \(x,y,h\ge0\),
\[
    \Psi(x+h,y)-\Psi(x,y)\ge \frac{x}{x+1}h,
    \qquad
    \Psi(x,y+h)-\Psi(x,y)\ge \frac{y}{y+1}h .
\]
\end{lemma}

The proof is a calculus argument, given in the appendix. 

\begin{lemma}
\label{lem:positive-H-variance}
The parameters $\sigma^2$ and $\sigma_H^2$ from \cref{thm:HJ CLT} are both strictly positive. 
\end{lemma}

\begin{proof}
For \(k\ge0\), let
\(\tree_k\) be a random BST on \(k\) nodes and, just as in the proof of \cref{thm:main-random}, set $H_k:=\sqrt{\beta(\tree_k)}$ and $V_k:=\pVar{H_k}$
with \(H_0=V_0=0\).

We first note that \(V_3>0\). Indeed, a random BST on three nodes has positive
probability to be balanced and positive probability to be a path. The corresponding \(\beta\)-values are distinct. For the balanced tree,
\[
    \beta_{\rm bal}=G(1,1)
    \le
    \frac{1+2\log 4}{3/5}.
\]
On the other hand, using the boundary formula for \(G(p,0)\) from
\cref{lem:G-boundary},
\[
    G(1,0)>\frac{157}{50},
    \qquad
    G\left(\frac{157}{50},0\right)>\frac{127}{20}.
\]
Thus, by monotonicity of $G$ in both of its arguments,
\[
    \beta_{\rm path}
    =
    G(G(1,0),0)
    >
    G\left(\frac{157}{50},0\right)
    >
    \frac{127}{20}
    >
    \frac{1+2\log 4}{3/5}
    \ge \beta_{\rm bal}.
\]
Hence \(H_3\) is not supported on a single point, and so \(V_3>0\). 

We use $V_3 > 0$ as a base case in the recurrence that we now develop for $V_k$. Let
\[
    \alpha_i:=\frac{\sqrt{i}}{\sqrt{i}+1},\qquad i\ge0.
\]
Note the trivial bound that for any fixed tree $S$, \(\beta(S)\ge |S|\). Hence
\(H_i\ge \sqrt{i}\) almost surely. Lemma~\ref{lem:sqrt-rec-lower-slope} 
implies that, for $x\geq y\geq \sqrt{i}$,
the map \(x\mapsto \Psi(x,y)\) has lower slope at least \(\alpha_i\),
uniformly in \(y\). The same statement holds in the second coordinate.
We need the elementary fact that if \(f\) is non-decreasing and $f(u)-f(v)\ge \alpha(u-v)$ for every $u \ge v$ in the support of a random variable \(X\), then
\[
    \pVar{f(X)}\ge \alpha^2\,\pVar{X}.
\]
Indeed, letting $X'$ be an independent copy of $X$, we have 
\begin{equation}\label{eq:var-ineq}\pVar{f(X)}
= \frac12\E{(f(X)-f(X'))^2}\ge \frac{\alpha^2}{2}\E{(X-X')^2} = \alpha^2\pVar{X}.
\end{equation}
So, let \(X:=H_i\) and \(Y:=H'_j\), where \(H'_j\) is an independent copy of
\(H_j\). Then, we claim:
\[
\begin{aligned}
    \pVar{\Psi(X,Y)}
    &=
    \E{\pVar{\Psi(X,Y)\mid Y}}
    +
    \pVar{\E{\Psi(X,Y)\mid Y}}  \\
    &\ge
    \alpha_i^2V_i+\alpha_j^2V_j .
\end{aligned}
\]
The equality follows from the law of total variance. In the inequality, the first term is lower bounded by $\alpha_i^2V_i$ using \eqref{eq:var-ineq} and \cref{lem:sqrt-rec-lower-slope}. Finally, for the second term, define \(g(y):=\E{\Psi(X,y)}\), where the expectation is
over \(X\). If \(y_2\ge y_1\) are in the support of \(Y\), then by \cref{lem:sqrt-rec-lower-slope}
\[
    g(y_2)-g(y_1)
    =
    \E{\Psi(X,y_2)-\Psi(X,y_1)}
    \ge
    \alpha_j(y_2-y_1).
\]
Hence, applying \eqref{eq:var-ineq} to \(g(Y)\) yields
\[
    \pVar{\E{\Psi(X,Y)\mid Y}}
    =
    \pVar{g(Y)}
    \ge
    \alpha_j^2V_j.
\]
The claimed lower bound on $\pVar{\Psi(X,Y)}$ follows. Next, let $L := L_\pi$ denote the left subtree of the random permutation $\pi$. Note that \(|L|\) is uniform on \(\{0,\ldots,k-1\}\). Moreover, conditional on
\(|L|=i\),
\[
    H_k\stackrel{d}{=}\Psi(H_i,H'_{k-1-i}),
\]
where the two random variables on the right are independent. Hence, by the law of total variance,
\begin{align*}
    V_k \ge
    \E{\pVar{H_k\mid |L|}}  &=
    \frac1k\sum_{i=0}^{k-1}
    \pVar{\Psi(H_i,H'_{k-1-i})} \\ &\ge
    \frac1k\sum_{i=0}^{k-1}
    \left(\alpha_i^2V_i+\alpha_{k-1-i}^2V_{k-1-i}\right) =
    \frac{2}{k}\sum_{i=0}^{k-1}\alpha_i^2V_i .
\end{align*}
Set $R_k:=\sum_{i=0}^k \alpha_i^2V_i$. Since \(V_3>0\), we have \(R_3>0\). The preceding inequality gives $V_k\ge 2R_{k-1}/k$.
Hence, for \(k\ge4\),
\[
    R_k
    =
    R_{k-1}+\alpha_k^2V_k
    \ge
    R_{k-1}\left(1+\frac{2\alpha_k^2}{k}\right).
\]
Using the inequality $2\alpha_m^2/m = 2/(\sqrt{m}+1)^2 \ge 2/m-4/m^{3/2}$, valid for $m\ge 1$, it holds for some absolute constants $c,c' > 0$,
\[
    \prod_{m=4}^k
    \left(1+\frac{2\alpha_m^2}{m}\right) \ge \prod_{m=4}^k\left(1+ \frac{2}{m}-\frac{4}{m^{3/2}}\right)
    \ge c k^2\,,
\]
and hence \(R_k\ge c'k^2\). Recalling $V_k \ge 2R_{k-1}/k$, it holds for some other absolute constant $c''>0$ that $V_k \ge c''k$
for all sufficiently large \(k\). Therefore
\[
    \liminf_{k\to\infty}\frac{V_k}{k}>0.
\]
Since \cref{thm:HJ CLT} gives \(V_k=\sigma_H^2k+o(k)\), it follows that
\(\sigma_H^2>0\). In particular, $\sigma^2 = 4c_\typ \sigma_H^2 > 0$.
\end{proof}

\section{Scaling limits: dynamic programs and additivity}\label{sec:scaling}

We recall for convenience the following putative definitions, which we shortly will prove are well-posed:
\begin{align*}
    c_\typ = \,\lim_{k \to \infty} \frac{\beta_k^\typ}{k^2},& \qquad \beta_k^{\typ} \,= \E[\pi \sim \mathrm{Unif}(\Pi_k)]{\beta(\pi)} \\
    c_+ = \lim_{k \to \infty} \frac{\beta_k^+}{k^2},& \qquad \beta_k^+ = \max_{\pi \in \Pi_k} \beta(\pi) \\
    c_- = \lim_{k \to \infty} \frac{\beta_k^-}{k^2},& \qquad \beta_k^- = \min_{\pi \in \Pi_k} \beta(\pi)
\end{align*}
We also define an auxiliary sequence for upper bounding $c_\typ$ and $\beta_k^\typ$: letting $\gamma_0 = 0$ and $\gamma_1 = 1$, 
\beq{eq:gamdef}
    \gamma_k = \frac{1}{k} \sum_{i=1}^k G\left(\gamma_{i-1},\,\gamma_{k-i}\right)\,.
\enq
The main result of this section establishes the existence of these putative limits and provides bounds on the limits in terms of finite expressions.

\begin{lemma}[Finite estimation of scaling constants]\label{lem:subsupadd}
    The constants $c_\typ$, $c_+$, and $c_-$ exist. Moreover, they admit the following estimates: for every $k \ge 1$, 
    \begin{enumerate}
        \item\label{it:betagamma} 
        \beq{eq:betagamma}
        \Big(\frac{2}{(k+1)(k+2)}\sum_{j=1}^k \sqrt{\beta_j^-}\Big)^2 \leq \Big(\frac{2}{(k+1)(k+2)}\sum_{j=1}^k \sqrt{\beta_j^\typ}\Big)^2 \leq c_{\typ} \leq \frac{\beta_k^{\typ}}{k^2} \leq \frac{\gamma_k}{k^2}.
        \enq
        \item\label{it:c+} 
        \begin{equation}\label{eq:cplus-ineqs}
             \frac{\beta_k^+}{(k+1)^2} \le c_+ \le \frac{\beta_k^+}{k^2}\,.
        \end{equation}

        \item\label{it:c-} 
        \begin{equation}\label{eq:cminus-ineqs}
             \frac14 \leq c_- \le \frac{\beta_k^-}{k^2}\,.
        \end{equation}
    \end{enumerate}
\end{lemma}

We begin by expressing $\beta_k^\pm$ with dynamic programs.

\subsection{The DP for finding the extreme values}

The following recursive formulas are the basis of \cref{lem:subsupadd}.

\begin{proposition}
\beq{eq:beta+}
    \beta_k^+ = \max_{i \in [k]}~G\left(\beta_{i-1}^+,\,\beta^+_{k-i}\right)\,, \quad \text{ and } \quad \beta^+_0 := 0\,.
\enq
and 
\beq{eq:beta-}
    \beta_k^- = \min_{i \in [k]}~G\left(\beta_{i-1}^-,\,\beta^-_{k-i}\right)\,, \quad \text{ and } \quad \beta^-_0 := 0\,.
\enq
\end{proposition}

\begin{proof}
We focus on \eqref{eq:beta+} as the proof of \eqref{eq:beta-} proceeds symmetrically. We have 
\begin{align*}\beta^+_k = \max_{\pi \in \Pi_k} \beta(\pi) &= \max_{i \in [k]}\max_{\pi \in \Pi_k: \pi(1) = i} \beta(\pi)\\ 
&=\max_{i \in [k]}\max_{\pi \in \Pi_k: \pi(1) = i}G(\beta(L(\pi)),\beta(R(\pi)))\,,
\end{align*}
where the last equality uses \eqref{eq:rec fixed}. Recall $G$ is non-decreasing in both arguments. Moreover, for each
$i\in[k]$, for every pair
\[
    (\sigma,\rho)\in \Pi_{i-1}\times \Pi_{k-i},
\]
there is some $\pi\in\Pi_k$ with $\pi(1)=i$ and $(L(\pi),R(\pi)) = (\sigma, \rho)$.
Thus the maximum over $\{\pi\in\Pi_k:\pi(1)=i\}$ is attained by choosing
$\sigma$ and $\rho$ to maximize $\beta$ on $\Pi_{i-1}$ and $\Pi_{k-i}$,
respectively. Therefore,

\begin{align*}\beta^+_k  
&=\max_{i \in [k]}\max_{\pi \in \Pi_k: \pi(1) = i}G\left(\beta_{i-1}^+,\beta_{k-i}^+\right) =\max_{i \in [k]}G\left(\beta_{i-1}^+,\beta_{k-i}^+\right)\,.\qedhere
\end{align*}
\end{proof}

\subsection{Proof of \cref{lem:subsupadd}} 

Let $G(p,q)$ be as defined in \eqref{eq:def-G1}. By evaluating the integrals in \eqref{eq:gamdef}, \eqref{eq:beta+}, and \eqref{eq:beta-}, we find that $\gamma_0 = \beta^+_0 = \beta^-_0 = 0$, $\gamma_1 = \beta^+_1 = \beta^-_1 = 1$, and
\beq{eq:eval}\gamma_k = \frac1k\sum_{j = 1}^k G(\gamma_{j-1}, \gamma_{k-j}), \qquad \beta^+_k = \max_{j \in [k]}G(\beta^+_{j-1}, \beta^+_{k-j}), \qquad  \beta^-_k = \min_{j \in [k]}G(\beta^-_{j-1}, \beta^-_{k-j}).\enq

\begin{proof}[Proof of existence of $c_-$, $c_+$ and $c_{\typ}$]
We now show that the constants $c_-$, $c_+$ and $c_{\typ}$ exist with a sub-additivity argument. Fix a permutation $\pi \in \Pi_k$, $\ell \in [k-1]$, and $\alpha \in (0,1)$. Consider the optimal embedding strategy for $\pi$ subject to the additional constraint that the lowest $\ell$ elements of $\pi$ are embedded in the interval $[0,\alpha]$ and the remaining elements of $\pi$ are embedded in the interval $[\alpha, 1]$. We claim that if $\alg^{\leq \ell}$ is an embedding strategy for $\pi^{\leq \ell}$, the permutation with the same order type as the sequence $(\pi(i))_{i:\pi(i) \leq \ell}$, and $\alg^{> \ell}$ is an embedding strategy for $\pi^{> \ell}$, the permutation with the same order type as the sequence $(\pi(i))_{i:\pi(i) > \ell}$, then
\beq{eq:osalph}
    \beta(\pi) \leq \alpha^{-1}\EE T_{\alg^{\leq \ell}} + (1 - \alpha)^{-1}\EE T_{\alg^{> \ell}}.
\enq
As the proof of \eqref{eq:osalph} is essentially the same as the proof of \cref{lemma:fix-rec-UB}, proceeding via a direct thinning/time-change argument, its proof is given in the appendix. Taking the infimum over embedding strategies for $\pi^{\leq \ell}$ and $\pi^{> \ell}$ in \eqref{eq:osalph} gives
\beq{eq:alph}
    \beta(\pi)
    \leq
    \alpha^{-1}\beta(\pi^{\leq \ell})
    +
    (1-\alpha)^{-1}\beta(\pi^{>\ell}).
\enq
We now prove subadditivity for the three sequences
$(\beta_k^\typ)_{k\ge1}$, $(\beta_k^+)_{k\ge1}$, and
$(\beta_k^-)_{k\ge1}$. Fix $m,n\ge1$ and apply \eqref{eq:alph} with $\alpha := m/(m+n)$. Then, every $\pi\in\Pi_{m+n}$ satisfies
\[
    \beta(\pi)
    \le
    \frac{m+n}{m}\,\beta(\pi^{\le m})
    +
    \frac{m+n}{n}\,\beta(\pi^{>m}).
\]

For $\beta^+$, choose $\pi$ maximizing $\beta$ on $\Pi_{m+n}$. Then:
\[
    \beta_{m+n}^+
    \le \frac{m+n}{m}\,\beta(\pi^{\le m})
    +
    \frac{m+n}{n}\,\beta(\pi^{>m}) \le 
    \frac{m+n}{m}\,\beta_m^+
    +
    \frac{m+n}{n}\,\beta_n^+.
\]

For $\beta^\typ$, let $\boldsymbol\pi$ be uniform on $\Pi_{m+n}$ and take
expectations in the preceding pointwise inequality. The standardized patterns
$\boldsymbol\pi^{\le m}$ and $\boldsymbol\pi^{>m}$ are uniform on $\Pi_m$ and
$\Pi_n$, respectively. Hence
\[
    \beta_{m+n}^\typ
    \le
    \frac{m+n}{m}\,\beta_m^\typ
    +
    \frac{m+n}{n}\,\beta_n^\typ.
\]

For $\beta^-$, choose $\sigma\in\Pi_m$ and $\rho\in\Pi_n$ attaining
$\beta_m^-$ and $\beta_n^-$. Let $\pi\in\Pi_{m+n}$ be any permutation such that
$\pi^{\le m}=\sigma$ and $\pi^{>m}=\rho$. Then
\[
    \beta_{m+n}^-
    \le
    \beta(\pi)
    \le
    \frac{m+n}{m}\,\beta_m^-
    +
    \frac{m+n}{n}\,\beta_n^-.
\]

Thus, for each of the three choices of sequence $\xi_k\in\{\beta_k^\typ,\beta_k^+,\beta_k^-\}$, we have that $(\xi_k/k)_k$ is subadditive:
\[
    \frac{\xi_{m+n}}{m+n}
    \le
    \frac{\xi_m}{m}+\frac{\xi_n}{n}.
\]
Fekete's lemma gives
\beq{eq:subcons}
    \lim_{k\to\infty}\frac{\xi_k}{k^2}
    =
    \inf_{k\ge1}\frac{\xi_k}{k^2}.
\enq
In particular, the three limits exist, and we have
\[
c_\typ\le \frac{\beta_k^\typ}{k^2},\qquad
c_+\le \frac{\beta_k^+}{k^2},\qquad
c_-\le \frac{\beta_k^-}{k^2}
\]
for every $k\ge1$. 
\end{proof}

We now turn to the remaining inequalities in \cref{lem:subsupadd}.

\begin{proof}[Proof of \eqref{eq:betagamma}] The lower bound was established in \eqref{eq:ctyp-LB}. For the upper bound, it remains to show that $\beta_k^{\typ}\le \gamma_k$ for all $k\ge0$. We proceed by induction. The cases $k=0,1$ hold by definition. Consider $k\ge2$. Using \eqref{eq:rec fixed} and the general inequality
\(\mathbb E[\inf_\theta Z_\theta]\le \inf_\theta \mathbb E[Z_\theta]\), it holds for fixed $i$ that conditional on $\pi(1)=i$,
\[
\mathbb E[\beta(\pi)\mid \pi(1)=i]
\le \inf_{0<a<b<1}
\cb{ \frac{1}{b-a}+\frac{1}{b-a}\int_a^b \frac{\mathbb E[\beta(L(\pi))\mid\pi(1)=i]}{x}+\frac{\mathbb E[\beta(R(\pi))\mid\pi(1)=i]}{1-x}\,dx }.
\]
But, conditionally on $i$, $L(\pi)$ and $R(\pi)$ are independent and uniform random permutations on $[i-1]$ and $[k-i]$, hence
\begin{align*}
\beta_k^{\mathrm{typ}}
\le
\frac{1}{k}\sum_{i=1}^k
G\!\left(
    \mathbb{E}_{\pi\in\Pi_k:\,\pi(1)=i}\bigl[\beta(L(\pi))\bigr],
    \mathbb{E}_{\pi\in\Pi_k:\,\pi(1)=i}\bigl[\beta(R(\pi))\bigr]
\right) 
=
\frac{1}{k}\sum_{i=1}^k
G\!\left(
    \beta_{i-1}^{\mathrm{typ}},
    \beta_{k-i}^{\mathrm{typ}}
\right).
\end{align*}

By the induction hypothesis and the coordinate-wise monotonicity of $G$,
\[
    \beta_k^{\mathrm{typ}}
    \le
    \frac{1}{k}\sum_{i=1}^k
    G(\gamma_{i-1},\gamma_{k-i})
    =
    \gamma_k.
\]
Combining with the subadditivity formula above yields the upper bound in \eqref{eq:betagamma}.
\end{proof}

\begin{proof}[Proof of \eqref{eq:cplus-ineqs}]
    We prove a ``shifted superadditivity'' principle to establish the remaining lower bound. Recall $\beta^+_{k} = \max_{j \in [k]}G(\beta^+_{j-1}, \beta^+_{k-j})$. Then for any $n,m \ge 0$, we obtain a lower bound on $\beta^+_{n+m+1}$ by selecting $j = n+1$:
    \[
        \beta^+_{n+m+1} \ge G(\beta^+_{n},\beta^+_m)\,.
    \]
    It then follows from \eqref{eq:G1} that
    \begin{equation}
        \sqrt{\beta^+_{n+m+1}} \ge \sqrt{\beta^+_n} + \sqrt{\beta^+_{m}}\,.
    \end{equation}
    In particular, the shifted sequence $(\beta_{k-1}^+)^{1/2}$ is \textit{super}additive. (More transparently, the sequence $a_r:=(\beta^+_{r-1})^{1/2}$, defined for $r\ge1$, is superadditive, since the preceding display gives $a_{r+s}\ge a_r+a_s$ with $r=n+1$ and $s=m+1$). Hence Fekete's lemma implies 
    \[
        \lim_{k\to \infty}\frac{\beta^+_{k-1}}{k^2} = \sup_{k \in \NN} \frac{\beta^+_{k-1}}{k^2}\,.
    \]    
    Thus
    \[
        \sup_{k \in \NN} \frac{\beta^+_{k-1}}{k^2} = \lim_{k \to \infty} \frac{\beta^+_{k-1}}{k^2}  = \lim_{k \to \infty}  \frac{\beta^+_{k-1}}{(k-1)^2} =: c_+,
    \]
    establishing the claimed lower bound in \eqref{eq:cplus-ineqs} after reindexing.
    \end{proof}

\begin{proof}[Proof of \eqref{eq:cminus-ineqs}]
It remains to prove
$c_-\ge1/4$. We prove by induction on $k\ge1$ that every $\pi\in\Pi_k$
satisfies
\[
    \beta(\pi)\ge \frac{(k+1)^2}{4}.
\]
The case $k=1$ is immediate. Let $k\ge2$, and let $L$ and $R$ be the left and
right subpatterns of $\pi$. If both are non-empty, then by the induction
hypothesis and \eqref{eq:G1},
\[
    \beta(\pi)
    \ge
    \left(\frac{|L|+1}{2}+\frac{|R|+1}{2}\right)^2
    =
    \frac{(k+1)^2}{4}.
\]
If exactly one subtree is empty, say $L=\emptyset$, then by the induction
hypothesis and \eqref{eq:G2},
\[
    \beta(\pi)
    \ge
    \left(\frac{|R|+1}{2}+\frac12\right)^2
    =
    \frac{(k+1)^2}{4},
\]
and the case $R=\emptyset$ is symmetric. This completes the induction. Taking $k\to\infty$ gives $c_-\ge1/4$.

\end{proof}

\section{Scaling limits: interval certificates}\label{sec:numer}
The main result of this section is the rigorous numerical evaluation of the bounds in \eqref{eq:betagamma}, \eqref{eq:cplus-ineqs}, and \eqref{eq:cminus-ineqs} for large finite $k$, organized as follows. First, we formally introduce so-called interval certificates as a formal proof system. Then, we give an analytic proof that the proposed elementary certificates are genuine certificates for the scalar kernel \(G\) and for the finite dynamic programs in \eqref{eq:eval}. Finally, we provide explicit candidate intervals for the outputs of these dynamic programs, as well as rigorous numerical verification, using the interval-arithmetic package \texttt{FLINT/arb}, of the corresponding elementary certificates for these interval containments. 

The code used for evaluating interval certificates, as well as the actual proposed intervals and certified slacks, are available at \url{https://github.com/dylanaltschuler/online-permutation-embedding/tree/main} and implemented in \texttt{Python}.

\subsection{Interval certificates}
Recall the following scalar kernel, arising from evaluating the integral in \eqref{eq:rec fixed},
\begin{equation}\label{eq:def-G}
G(p,q)
:=
\inf_{0<a<b<1}
\frac{1+p\log(b/a)+q\log((1-a)/(1-b))}{b-a},
\qquad p,q\ge0\,.
\end{equation}
The key observation enabling our numerical evaluations is \eqref{eq:eval}, which implies that $\beta_k^{\pm}$ and $\gamma_k$ for $k \in [k_0]$ are expressible by finitely many compositions of \(G\) and basic operations (arithmetic, averages, maxima, and minima). Hence, our main task is to understand how to rigorously compute $G$ to high accuracy. 

\begin{remark}[Elementary operations and scalar certificates]
    In what follows, \textit{elementary functions} are functions which are computable with rigorous error guarantees using standard computing packages. Concretely, we use rigorous interval arithmetic library \texttt{FLINT}/\texttt{Arb}\footnote{Documentation: \url{https://flintlib.org/doc/arb.html}} \cite{arb}, wrapped in Python by \texttt{python-flint}. This library provides computation with rigorous error guarantees for arithmetic operations, the logarithm, and the exponential. For our purposes, it suffices to define elementary functions as finite compositions of these three primitives, such as square-roots and hyperbolic trigonometric functions. 
    
    A scalar \textit{elementary inequality} is the task of determining whether an elementary function evaluated on a specific input is non-negative. (``Scalar'' is in reference to evaluation of the function at a specific input, as opposed to the non-elementary task of proving function-wise inequalities). Finally, an \textit{elementary certificate} is a claim together with a finite, explicitly listed collection of scalar elementary inequalities whose validity implies the claim. 
\end{remark}
We now introduce a class of ``checkable proofs'' which we call \textit{interval certificates}: they assert ``$G(p,q)$ lies in an interval'' and consist only of checking a bounded number of scalar elementary inequalities. Specifically, we will only need to check the signs of the following residual functions: 

\begin{definition}[Residuals]
     For \(q>0\) set
    \begin{equation}\label{eq:boundary-residual}
    \mathsf B_q(u)
    :=
    \sqrt q\,u-q\log\pa{1+\frac{u}{\sqrt q}}-1,
    \qquad u\ge0,
    \end{equation}
    and for \(p,q>0\) set
    \begin{align}
    \mathsf P_{p,q}(u)
    &:=
    2\sqrt{pq}\sinh u-(p+q)u
     +(p-q)
     \log\frac{\sqrt p\,e^u+\sqrt q}{\sqrt p+\sqrt q\,e^u},
    \label{eq:interior-root-functional}\\
    \mathsf I_{p,q}(u)
    &:=
    \frac{p+q+2\sqrt{pq}\cosh u}{2\sqrt{pq}\sinh u}
    \pa{\mathsf P_{p,q}(u)-1},
    \qquad u>0 .
    \label{eq:interior-residual}
    \end{align}
\end{definition}

\begin{lemma}[Interval certificates for \(G\)]\label{lem:G-cert} The functions \(\mathsf B_q\), \(\mathsf P_{p,q}\), and \(\mathsf I_{p,q}\)
are elementary on their stated domains. Moreover, for \(p,q\ge0\), we have the following (elementary) interval certificates on $G$:
\begin{enumerate}
    \item If \(p=q=0\), then \(G(p,q)\in[1,1]\).
    \item If \(p=0<q\), and if \(0\le \underline{u}\le \overline{u}\) satisfy
    \[
        \mathsf B_q(\underline{u})\le0\le\mathsf B_q(\overline{u}),
    \]
    then
    \[
        G(0,q)\in\br{q+\sqrt q\,\underline{u},\ q+\sqrt q\,\overline{u}}.
    \]
    The case \(q=0<p\) is symmetric.
    \item If \(p,q>0\), and if \(0<\underline{u}\le \overline{u}\) satisfy
    \[
        \mathsf I_{p,q}(\underline{u})\le0\le\mathsf I_{p,q}(\overline{u}),
    \]
    then
    \[
        G(p,q)
        \in
        \br{
        p+q+2\sqrt{pq}\cosh \underline{u},
        p+q+2\sqrt{pq}\cosh \overline{u}}.
    \]
\end{enumerate}
\end{lemma}

The proof is briefly deferred. Next, we develop notation for propagation of interval estimates.

\begin{definition}[Interval extension]
    For the remainder of this section, an \textit{interval} is a closed interval with rational endpoints. An \textit{interval enclosure} of a variable $x$ is an interval $[x] := [\underline{x},\overline{x}] \ni x$. Its width is \(\operatorname{wd}([x])=\overline{x}-\underline{x}\). 
    For a real function $F(x_1,\dots,x_m)$, a function $[F]([x_1],\dots,[x_m])$ is said to be an \textit{interval extension} of $F$ if
    \[
        x_j\in[x_j]\quad (1\le j\le m)
        \qquad\Longrightarrow\qquad
        F(x_1,\dots,x_m)\in [F]([x_1],\dots,[x_m]).
    \]
    Interval extensions also naturally define operations on intervals. For example,
    \[
    I+J=[\underline I+\underline J,\overline I+\overline J],
    \qquad
    I/c=[\underline I/c,\overline I/c]\quad(c>0).
    \]
    We henceforth implicitly assume all interval operations are performed with ``outward'' rounding (lower endpoints rounded downward and upper endpoints upward) to maintain rationality of the endpoints.     
\end{definition}

Crucially, our kernel $G$ is monotone in its inputs, allowing for a natural interval extension.

\begin{proposition}[Monotonicity of $G$]
\label{prop:G-monotone} The map
$(p,q) \mapsto G(p,q)$ is non-decreasing in each coordinate on $[0,\infty)^2$. In particular, the following is a valid interval extension of $G$. 
\begin{equation}\label{eq:G-interval-extension}
[G]([p],[q])
:=
\br{G(\underline p,\underline q),\ G(\overline p,\overline q)},
\end{equation}
\end{proposition}

\begin{proof}
By definition of $[p]$ and $[q]$ as interval enclosures, we may assume they have rational endpoints. For fixed $0<a<b<1$, the objective in \eqref{eq:def-G} is affine and increasing in $p$ and $q$.  Taking the infimum over $(a,b)$ preserves coordinate-wise monotonicity.
\end{proof}

With our interval extensions and certificates in hand for $G$, we turn to
constructing interval extensions and certificates for the full dynamic program.

\begin{corollary}[Interval certificates for the DP]\label{cor:DP-certificate}
Fix \(k_0\).  Suppose we are given rational intervals
\[
[\gamma_k],\qquad [\beta_k^+],\qquad [\beta_k^-]\qquad (0\le k\le k_0).
\]
Suppose the boundary conditions
$
[\gamma_0]=[\beta_0^+]=[\beta_0^-]=[0,0]$ and $
[\gamma_1]=[\beta_1^+]=[\beta_1^-]=[1,1]$. Define $
[G]([\beta_{i-1}^{\pm}],[\beta_{k-i}^{\pm}])
=:[\underline G^{\pm}_{i,k},\overline G^{\pm}_{i,k}]$ and further suppose that for each \(2\le k\le k_0\), we have:
\begin{align}
\frac1k\sum_{i=1}^k [G]([\gamma_{i-1}],[\gamma_{k-i}])
&\subseteq [\gamma_k],
\label{eq:gamma-cert-inclusion}\\
\br{\max_{1\le i\le k}\underline G^+_{i,k},
      \max_{1\le i\le k}\overline G^+_{i,k}}
&\subseteq [\beta_k^+],
\label{eq:beta-plus-cert-inclusion}\\
\br{\min_{1\le i\le k}\underline G^-_{i,k},
      \min_{1\le i\le k}\overline G^-_{i,k}}
&\subseteq [\beta_k^-],
\label{eq:beta-minus-cert-inclusion}
\end{align}
Then \([\gamma_k]\), \([\beta_k^+]\), and \([\beta_k^-]\) contain the exact values
\(\gamma_k\), \(\beta_k^+\), and \(\beta_k^-\), respectively, for every
\(0\le k\le k_0\). 
\end{corollary}

\begin{proof}
We argue by induction on \(k\). The cases \(k=0,1\) hold by the stated
boundary conditions. Assume the claim holds for all smaller indices. For each
\(1\le i\le k\), the induction hypothesis and the coordinate-wise monotonicity
of \(G\) imply that
\[
G(\gamma_{i-1},\gamma_{k-i})
\in [G]([\gamma_{i-1}],[\gamma_{k-i}]),
\]
and similarly for the \(G^+\) and \(G^-\) sequences. Taking sums, maxima, and
minima of the corresponding interval enclosures preserves containment, and hence the displayed inclusions imply the desired containment for
\(\gamma_k\), \(\beta_k^+\), and \(\beta_k^-\).
\end{proof}

Note in particular that, in view of the elementary certificates for \(G\)
constructed in \cref{lem:G-cert}, each inclusion in \eqref{eq:gamma-cert-inclusion}--\eqref{eq:beta-minus-cert-inclusion} reduces to finitely many scalar endpoint inequalities. Thus \cref{cor:DP-certificate}, together with the certified calls to \(G\), gives an elementary certificate for the finite dynamic programs.

\subsection{Numerical evaluation of the certificates}

We now isolate all of the actual numerical computations. 

\begin{proposition}[Numerical validation of certificates]
    Let $k_0 = 5000$. The interval enclosures \( [\gamma_k]\), \([\beta_k^+]\), and \([\beta_k^-]\), for \(0\le k\le k_0\), supplied in \texttt{intervals.json} are valid interval enclosures. In particular, \cref{thm:numerical-scaling} holds.
\end{proposition}

\begin{proof}

    The verification script \texttt{certify.py} reads the candidate intervals supplied in \texttt{intervals.json}, evaluates the residual inequalities from
    \cref{lem:G-cert} using \texttt{FLINT/arb} with outward rounding, and checks
    the interval inclusions in \cref{cor:DP-certificate}. The computed slacks for the certificate inequalities are recorded in \texttt{slacks.json} and are all positive. Hence the interval enclosures
    contain the exact values of \(\gamma_k\), \(\beta_k^+\), and \(\beta_k^-\) for
    all \(0\le k\le k_0\). This proves that our candidate interval enclosures are valid interval enclosures. 
    
    Examining the certified endpoint values of the (now certified) interval enclosures and applying \cref{lem:subsupadd} yields all numerical estimates in \cref{thm:numerical-scaling}, except for the analytic lower bound \(c_-\ge1/4\) which was already proved in \cref{lem:subsupadd}.
\end{proof}

\begin{remark}[Implementation details]
We provide some additional details for the interested reader. First, recall that we are free to generate the candidate intervals arbitrarily; only the verification of the corresponding interval certificate inequalities needs to be done rigorously. As such, the candidate intervals were generated as follows. For $k = 0,1$, the intervals are exactly given by the boundary conditions. For \(k\ge2\), the respective DPs are solved using high-precision (uncertified) computation and then padded with intervals of width approximately \(10^{-9}\). The case \(k=2\) produces the smallest interval widths and hence the most delicate certificate inequalities.
    
These heuristic interval candidates were then rigorously validated by checking the interval certificates defined in \cref{cor:DP-certificate} using \texttt{FLINT/arb}. The minimum recorded slack over all validated certificate inequalities is positive, of order \(10^{-9}\), and occurs in the checks for \(k=2\). Here, our definition of ``slack'' already takes into account the error estimates guaranteed by \texttt{arb}; for example, if \texttt{FLINT/arb} computes $f(x) \in [a,b]$ for some $a > 0$, then the slack of the inequality $f(x) \ge 0$ is defined as $a$. That is, any positive slack already constitutes rigorous verification of an inequality, and is not subject to further numerical error. Thus, a positive lower endpoint for the \texttt{arb} enclosure of a residual
rigorously certifies the corresponding inequality. 
\end{remark}

\begin{remark}[Redundancy]
    As a check against implementation errors in \texttt{python-flint}/\texttt{arb},
we separately verified the same certificate inequalities using
\texttt{SageMath/mpfi} and directed-rounding \texttt{SageMath/mpfr} \cite{sage}, both implemented in \texttt{C}. We present the \texttt{FLINT} verification because it is the fastest of the three. All three implementations independently verify the candidate intervals with positive slack.
\end{remark}

We now turn to the proof of \cref{lem:G-cert}. The main technical component is isolated in the following two lemmas, which analyze the first-order optimality condition for selecting $(a,b)$ in the variational problem defining $G(p,q)$, treating separately the interior case \(p,q>0\) from the boundary cases
where one of \(p,q\) is zero.

\begin{lemma}[Boundary inputs]\label{lem:G-boundary}
Let \(q>0\).  There is a unique \(u_q>0\) satisfying
\[
    \mathsf B_q(u_q)=0,
\]
and
\[
    G(0,q)=q+\sqrt q\,u_q.
\]
The symmetric statement holds for \(G(p,0)\), and \(G(0,0)=1\).
\end{lemma}

\begin{proof}
Fix \(q>0\), and let
\[
    s:=\frac{1-a}{1-b}>1.
\]
Then \(b=1-(1-a)/s\) and
\[
    b-a=\frac{(s-1)(1-a)}{s}.
\]
Then for fixed \(s\), \(G(0,q)\) is the infimum over $a$ of
\[
    \frac{s}{(s-1)(1-a)}\pa{1+q\log s},
\]
which decreases as $a$ approaches $0$ from above.  Thus
\[
G(0,q)=\inf_{s>1}\Psi_q(s),
\qquad
\Psi_q(s):=\frac{s}{s-1}\pa{1+q\log s}.
\]
A direct derivative computation gives
\[
\Psi_q'(s)=\frac{q(s-1)-q\log s-1}{(s-1)^2}.
\]
The numerator is strictly increasing on \((1,\infty)\), starts at \(-1\), and tends
to infinity.  Hence \(\Psi_q\) has a unique minimizer \(s_q>1\), characterized by
\[
    s_q-\log s_q=1+\frac1q.
\]
Then
\(\Psi_q(s_q)=q s_q\).  Writing \(s_q=1+u_q/\sqrt q\) gives exactly
\(\mathsf B_q(u_q)=0\) and \(G(0,q)=q+\sqrt q\,u_q\). Moreover, \(\mathsf B_q(0)=-1\), \(\mathsf B_q(u)\to\infty\) as
\(u\to\infty\), and
\(\mathsf B_q'(u)=u/(1+u/\sqrt q)>0\) for \(u>0\). Hence the root exists and is
unique.
\end{proof}

\begin{lemma}[Interior inputs]\label{lem:G-interior}
Let \(p,q>0\).  There is a unique \(u_{p,q}>0\) satisfying
\[
    \mathsf P_{p,q}(u_{p,q})=1,
\]
and
\[
    G(p,q)=p+q+2\sqrt{pq}\cosh u_{p,q}.
\]
Moreover,
\[
    \mathsf P_{p,q}'(u)
    =
    \frac{4pq\sinh^2u}{p+q+2\sqrt{pq}\cosh u}>0
    \qquad(u>0).
\]
\end{lemma}

\begin{proof}
Let \(\alpha=\sqrt p\) and \(\delta=\sqrt q\).  For \(u>0\), define
\[
    a(u)=\frac{\alpha}{\alpha+\delta e^u},
    \qquad
    b(u)=\frac{\alpha e^u}{\alpha e^u+\delta}.
\]
Then \(0<a(u)<b(u)<1\), and both \(a(u)\) and \(b(u)\) solve
\[
    \lambda=\frac p x+\frac q{1-x}
\]
with
\[
    \lambda=p+q+2\sqrt{pq}\cosh u.
\]
Substituting these two roots into
\[
\lambda(b-a)
-p\log\frac ba
-q\log\frac{1-a}{1-b}
\]
gives exactly \(\mathsf P_{p,q}(u)\).  Thus the first-order conditions for the
interior minimizer of \eqref{eq:def-G} reduce to \(\mathsf P_{p,q}(u)=1\).
The objective diverges on the boundary when \(p,q>0\), so the minimizer is indeed
interior.

Finally, differentiating after this parametrization gives
\[
    \mathsf P_{p,q}'(u)
    =
    \left(b(u)-a(u)\right)\overbrace{\frac{d}{du}\pa{p+q+2\sqrt{pq}\cosh u}}^{2\sqrt{pq}\sinh u}
    =
    \frac{4pq\sinh^2u}{p+q+2\sqrt{pq}\cosh u}>0.
\]
Since \(\mathsf P_{p,q}(u)\to0\) as \(u\to 0^+\), and
\(\mathsf P_{p,q}(u)\to\infty\) as \(u\to\infty\), the equation
\(\mathsf P_{p,q}(u)=1\) has a unique solution.  The formula for \(G(p,q)\)
follows.
\end{proof}

\begin{proof}[Proof of \cref{lem:G-cert}]
The case \(p=q=0\) is immediate.  In the boundary case, \cref{lem:G-boundary}
shows that the exact value is \(q+\sqrt q\,u_q\), where \(u_q\) is the unique root
of \(\mathsf B_q\).  Since \(\mathsf B_q\) is increasing, the certified inequalities
force \(u_q\in[\underline u, \overline u]\).

In the interior case, \cref{lem:G-interior} shows that the exact value is
\(p+q+2\sqrt{pq}\cosh u_{p,q}\), where \(u_{p,q}\) is the unique root of
\(\mathsf P_{p,q}(u)-1\).  The factor multiplying \(\mathsf P_{p,q}(u)-1\) in
\eqref{eq:interior-residual} is positive for \(u>0\), so the signs of
\(\mathsf I_{p,q}\) and \(\mathsf P_{p,q}-1\) agree.  The certified inequalities
therefore imply \(u_{p,q}\in[\underline u, \overline u] \), and the claimed interval for \(G(p,q)\)
follows because \(u\mapsto p+q+2\sqrt{pq}\cosh u\) is increasing on \([0,\infty)\).
\end{proof}

\section*{Acknowledgments}
KT is partially supported by the NSF grant DMS 2331037. ChatGPT 5.5 Pro was used to prepare the figures and computational supplement. The idea of using the CLT from \cite{tree-CLT}, parts of the martingale construction in Proposition~\ref{prop:recursive-value-lower-bound}, and simplifications in several calculus computations arose in conversation with the model. ChatGPT was also used for routine editing and checking. The vast majority of the manuscript was completed before involving AI assistance.

\appendix 
\section{Deferred Proofs}

\begin{proof}[Proof of \eqref{eq:osalph}]
    Let $\alg^{\leq \ell} = (\tau^L_j)_{j=1}^{\ell}$ be an embedding strategy for $\pi^{\leq \ell}$ and $\alg^{> \ell} = (\tau^R_j)_{j=1}^{k - \ell}$ be an embedding strategy for $\pi^{> \ell}$. We define an embedding strategy $(\tau_i)_{i = 1}^k$ with the desired property.

    Note that we may assume $\EE[T_{\alg^{\leq \ell}}]<\infty$ and $\EE[T_{\alg^{> \ell}}]<\infty$ since otherwise the lemma follows by taking $(\tau_i)_{i = 1}^k$ as the trivial strategy, namely $\tau_i = \min\{j > \tau_{i-1}: X_j \in ((\pi(i)-1)/k, \pi(i)/k)\}$, which has expected embedding time of $k^2$.

    Now, since the extra randomness $X_0$ is a countable collection of iid variables, we can write $X_0 = Y_L \sqcup Y_R$ where $Y_L$ and $Y_R$ are independent and distributed the same as $X_0$. We work with the random sets $\Omega_L =\{i: X_i < \alpha\}$ and $\Omega_R = \{i: X_i > \alpha\}$, noting that $\Omega_L \cup \Omega_R = \NN$ with probability one. 
    
    Suppose that $\tau_1,\dots, \tau_i$ have been defined for some $i < k$, and assume that $\pi(i + 1) \leq \ell$ (if not, the discussion is the same with $L$ replaced by $R$, $\pi^{\leq \ell}$ replaced by $\pi^{> \ell}$, and $\alpha$ replaced by $1 - \alpha$). Moreover, suppose that up to this point we have constructed two independent lists of iid $[0,1]$ variables $\cS^L_i$ and $\cS^R_i$ which are of finite length a.s.; at the start, we let $\cS^L_1 =(Y_L)$ and $\cS^R_1 = (Y_R)$. Let $m$ be the index in $\pi^{\leq \ell}$ corresponding to $\pi(i+1)$, and let $g_m$ be the sequence of Borel decision functions defining $\tau^L_m$, as in \cref{sec:ov}. 
    Letting `+' be concatenation of two sequences, we define
    \[\tau_{i+1} = \min\{t: g_m(\cS^L_i + (X_j/\alpha)_{j \in \Omega_L \cap [\tau_i + 1, t]}) = 1\},\]
    and we set $\cS^L_{i+1} = \cS^L_i + (X_j/\alpha)_{j \in \Omega_L \cap [\tau_i + 1, \tau_{i+1}]}$ and $\cS^R_{i+1} = \cS^R_i$. Finally, once $\tau_k$, $\cS^L_k$, and $\cS^R_k$ have been defined, we let $\cS^L = \cS^L_k + (X_j/\alpha)_{j \in \Omega_L\cap [\tau_k + 1, \infty)}$ and $\cS^R = \cS^R_k + ((X_j-\alpha)/(1 - \alpha))_{j \in \Omega_R\cap [\tau_k + 1, \infty)}$. It is clear by construction that for each $i \in [k]$, $\tau_i$ is a stopping time for the sequence $X_0,X_1,\dots$, so it remains to understand $\EE\tau_k$.

    Since $\tau_k = \sum_{i=0}^{k-1} \tau_{i+1} - \tau_i$ (setting $\tau_0 = 0$ for convenience), it suffices by linearity of expectation to understand $\EE[\tau_{i+1} - \tau_i]$. Fixing $i$ as before and recalling that $m$ is the index in $\pi^{\leq \ell}$ corresponding to $\pi(i+1)$, we have
    \[\tau_{i+1} - \tau_i = |\Omega_L \cap [\tau_i + 1, \tau_{i+1}]| + |\Omega_R \cap [\tau_i + 1, \tau_{i+1}]|= \tau^L_m - \tau^L_{m-1} + |\Omega_R \cap [\tau_i + 1, \tau_{i+1}]|.\]
    Let $Z_1,Z_2,\dots$ be an iid collection of geometric variables with parameter $\alpha$. Then 
    \[|\Omega_R \cap [\tau_i + 1, \tau_{i+1}]| = \sum_{j=1}^{\tau^L_m - \tau^L_{m-1}} (Z_j - 1),\]
    so Wald's identity gives
    \[\EE |\Omega_R \cap [\tau_i + 1, \tau_{i+1}]| = \EE[\tau^L_m - \tau^L_{m-1}]\EE[Z_1 - 1].\]
    Together with the previous displayed line, we find that 
    \[\EE [\tau_{i+1} - \tau_i] = \EE[\tau^L_m - \tau^L_{m-1}] + \EE|\Omega_R \cap [\tau_i + 1, \tau_{i+1}]| =  \EE[\tau^L_m - \tau^L_{m-1}] + \EE[\tau^L_m - \tau^L_{m-1}]\EE[Z_1 - 1] = \EE[\tau^L_m - \tau^L_{m-1}]\EE[Z_1].\]
    Since $\EE Z_1 = 1/\alpha$, we find that 
    \[\EE [\tau_{i+1} - \tau_i] = \EE[\tau^L_m - \tau^L_{m-1}]/\alpha.\]
    The same discussion holds for any $i$ such that $\pi(i+1) \leq \ell$, so
    \[\sum_{i: \pi(i+1) \leq \ell}\EE [\tau_{i+1} - \tau_i] = \sum_{m = 1}^{\ell} \EE[\tau^L_m - \tau^L_{m-1}]/\alpha = \frac{\EE\,T_{\alg^{\leq \ell}}}{\alpha}.\]
    Similarly,
    \[\sum_{i: \pi(i+1) > \ell} \EE [\tau_{i+1} - \tau_i]=  \frac{\EE\,T_{\alg^{> \ell}}}{1 - \alpha}.\]
    Thus
    \[\EE\tau_k = \sum_{i=1}^{k-1} \EE [\tau_{i+1} - \tau_i] = \frac{\EE\,T_{\alg^{\leq \ell}}}{\alpha} +\frac{\EE\,T_{\alg^{> \ell}}}{1 - \alpha}.\]
    Conditioning on $X_0$, we may fix a seed value for which the conditional expected run-time is no larger than the displayed expectation. This gives a non-random embedding strategy with the same bound. Since $\beta(\pi)$ is the infimum expected embedding time over all such strategies, the result follows.
\end{proof}

\begin{proof}[Proof of Lemma \ref{lem:sqrt-rec-lower-slope}]
    By symmetry it suffices to prove the first inequality. The case \(x=0\)
follows from monotonicity of \(G_{p,q}(a,b)\) in \(p\), so assume \(x>0\). 

    For each $h$, let $(a,b) = (a_h, b_h)$ achieving the minimum in \eqref{eq:def-G1} for $G((x + h)^2, y^2)$ so that $G((x + h)^2, y^2) = G_{(x + h)^2, y^2}(a,b)$. Then since $G(x^2,y^2) \leq  G_{x^2, y^2}(a,b)$, we have
    \[
        \lim_{h\rightarrow 0^+} \frac{\Psi(x+h,y) - \Psi(x,y)}{h} \geq \lim_{h\rightarrow 0^+} \frac{\sqrt{G_{(x + h)^2, y^2}(a,b)} - \sqrt{G_{x^2, y^2}(a,b)}}{h}.
    \]
    For small $h$, we have 
    \begin{align*}
    \sqrt{G_{(x + h)^2, y^2}(a,b)} &= \sqrt{\frac{1+(x^2 + 2xh + h^2)\log(b/a)+y^2\log((1-a)/(1-b))}{b-a}}\\
    &> \sqrt{G_{x^2, y^2}(a,b)} + \frac{(xh - O(xh^2))\log(b/a)}{(b-a)\sqrt{G_{x^2, y^2}(a,b)}}
    \end{align*}
    so 
    \[\lim_{h\rightarrow 0^+} \frac{\sqrt{G_{(x + h)^2, y^2}(a,b)} - \sqrt{G_{x^2, y^2}(a,b)}}{h}  = \lim_{h\rightarrow 0^+} \frac{(xh - O(xh^2))\log(b/a)}{h(b-a)\sqrt{G_{x^2, y^2}(a,b)}} = \lim_{h\rightarrow 0^+} \frac{x\log(b/a)}{(b-a)\sqrt{G_{x^2, y^2}(a,b)}}.\]
    Now $a(h)$ and $b(h)$ are continuous functions of $h$ and $\frac{\log(b/a)}{(b-a)\sqrt{G_{x^2, y^2}(a,b)}}$ is a continuous function of $a$ and $b$ at $h=0$ (using $a > 0$ since $x >0$), so
    \beq{eq:sqrt-slope-logmean}
        \lim_{h\rightarrow 0^+} \frac{\Psi(x+h,y) - \Psi(x,y)}{h} \geq \frac{x\log(b/a)}{(b-a)\Psi(x,y)} \geq \frac{2x}{(a+b)\Psi(x,y)}.
    \enq
    It remains to bound $\frac{2x}{(a+b)\Psi(x,y)}$

Put
\[
    v:=\Psi(x,y)^2,
    \qquad
    \phi(t):=\phi_{x^2,y^2}(t).
\]
By \cref{prop:F-root},
\begin{equation}\label{eq:sqrt-slope-positive-part}
    1=\int_0^1 (v-\phi(t))_+\,dt .
\end{equation}
Let \(a\) and \(b\) be the endpoints of the interval $\{t\in[0,1]:\phi(t)<v\}.$
Since \(x>0\), we have \(a>0\). Restricting
\eqref{eq:sqrt-slope-positive-part} to this interval gives
\begin{equation}\label{eq:sqrt-slope-level}
    1=\int_a^b (v-\phi(t))\,dt .
\end{equation}

We claim that
\begin{equation}\label{eq:sqrt-slope-endpoints}
    x^2=vab,
    \qquad
    y^2=v(1-a)(1-b).
\end{equation}
If \(y>0\), these follow by solving \(\phi(a)=\phi(b)=v\). If \(y=0\), then
\(b=1\), and the identities reduce to \(x^2=va\) and
\(0=v(1-a)(1-b)\). Hence, for \(a<t<b\),
\begin{equation}\label{eq:sqrt-slope-gap}
    v-\phi(t)
    =
    v\,\frac{(t-a)(b-t)}{t(1-t)},
\end{equation}
where in the case \(y=0\), \(b=1\), the right-hand side is interpreted after
canceling the removable factor \(1-t\).

It remains to bound \(a+b\). Combining
\eqref{eq:sqrt-slope-level} and \eqref{eq:sqrt-slope-gap},
\[
    1
    =
    v\int_a^b
        \frac{(t-a)(b-t)}{t(1-t)}
    \,dt
    \ge
    v\int_a^b
        \frac{(t-a)(b-t)}{t}
    \,dt .
\]
With the substitution \(t=s^2\),
\[
\begin{aligned}
    \int_a^b\frac{(t-a)(b-t)}{t}\,dt
    &=
    2\int_{\sqrt a}^{\sqrt b}
        (s-\sqrt a)(\sqrt b-s)
        \frac{(s+\sqrt a)(\sqrt b+s)}{s}
    \,ds  \\
    &\ge
    2(\sqrt b-\sqrt a)
    \int_{\sqrt a}^{\sqrt b}
        (s-\sqrt a)(\sqrt b-s)
    \,ds  \\
    &=
    \frac{(\sqrt b-\sqrt a)^4}{3}.
\end{aligned}
\]
Thus
\[
    \Psi(x,y)(\sqrt b-\sqrt a)^2\le \sqrt3<2.
\]
Using \(x^2=vab\), equivalently \(x=\Psi(x,y)\sqrt{ab}\), we get
\[
\begin{aligned}
    \Psi(x,y)(a+b)
    &=
    \Psi(x,y)
    \left(2\sqrt{ab}+(\sqrt b-\sqrt a)^2\right)  \\
    &=
    2x+\Psi(x,y)(\sqrt b-\sqrt a)^2
    \le
    2(x+1).
\end{aligned}
\]
Together with
\eqref{eq:sqrt-slope-logmean}, this implies
\[
    \lim_{h\rightarrow 0^+} \frac{\Psi(x+h,y) - \Psi(x,y)}{h} 
    \ge
    \frac{x}{\Psi(x,y)}\cdot\frac{2}{a+b}
    \ge
    \frac{x}{x+1}.
\]
Applying this derivative bound with \(x\) replaced by \(u\), and integrating
over \(u\in[x,x+h]\) (noting absolute continuity of $\Psi$ in this interval), gives
\[
    \Psi(x+h,y)-\Psi(x,y)
    =
    \int_x^{x+h}
        \frac{\partial}{\partial u}\Psi(u,y)
    \,du
    \ge
    \int_x^{x+h}\frac{u}{u+1}\,du
    \ge
    \frac{x}{x+1}h .
\]
\end{proof}

\subsection{Embedding into random permutations}

Let $\sigma \in \Pi_n$ be a random permutation, and let $X_0$ denote a random variable, or a collection of random variables independent from $\sigma$. A permutation stream embedding of $\pi \in \Pi_k$ in $\sigma$ is a sequence of stopping times $1\leq \tau_1\leq\tau_2\leq\dots\leq\tau_k \leq n+1$ w.r.t.\ the
sequence $X_0,\sigma(1),\sigma(2)\dots, \sigma(n)$, such that either $(X_{\tau_j})_{j=1}^k$ is
order-isomorphic to $\pi\in \Pi_k$ or $\tau_k = n+1$. We sometimes write $\alg^{\sigma}$ to denote a stream embedding algorithm in a random permutation and $T_{\alg^\sigma}$ as the last stopping time of $\alg^\sigma$.
Let $\pbeta(\pi)$ be the smallest $n$ such that there is a stream embedding of $\pi$ in a random permutation of length $n$ satisfying $\PP{\tau_k \leq n} \geq 1/2$. Our goal is \cref{prop:Stopi}, which states that
    \[\pbeta(\pi) \le (1 + o_k(1))\beta(\pi).\]
\Cref{prop:Stopi} follows from an elegant coupling, introduced in \cite{permutation}, which shows that embedding in a random stream is at least as hard as embedding in a random permutation. We need the notion of a continuous finite stream embedding of $\pi \in \Pi_k$: suppose $X_1,\dots, X_n$ are iid uniform $[0,1]$ variables, and $X_0$ is a random variable, or a collection of random variables independent from $X_1,\dots, X_n$. A finite stream embedding of $\pi$ is a sequence of stopping times $1\leq \tau_1 \leq\tau_2\leq\dots\leq\tau_k \leq n+1$ w.r.t.\ the
sequence $X_0,X_1,\dots, X_n$, such that either $(X_{\tau_j})_{j=1}^k$ is
order-isomorphic to $\pi\in \Pi_k$ or $\tau_k = n+1$. We use $\alg^n$ to denote a finite stream embedding algorithm over the sequence $X_0,X_1,\dots, X_n$ and $T_{\alg^n}$ as the last stopping time of $\alg^n$. Let $\beta^f(\pi)$ be the smallest $n$ such that there is a continuous finite stream embedding of $\pi$ satisfying $\PP{\tau_k \leq n} \geq 1/2$.
\begin{lemma}\label{lem:couppiS}
    For any $\pi \in \Pi_k$,
    \[\pbeta(\pi) \leq \beta^f(\pi).\]
\end{lemma}
\begin{proof}[Proof of \cref{lem:couppiS}]
    We use the coupling introduced in \cite{permutation} which allows for the transfer of a finite stream embedding algorithm to a permutation stream embedding. The coupling is based on the fact that if $(e_i)_{i=1}^{n+1}$ is an iid collection of exponential mean one variables, then the sequence $(Y_i)_{i=1}^n$ defined by 
    \[Y_i = \frac{e_1 + \dots + e_i}{e_1 + \dots + e_{n+1}}\]
    has the same distribution as the order statistics of $n$ iid uniform $[0,1]$ variables. Thus we can sample $n$ iid uniform $[0,1]$ variables $X_1,\dots, X_n$ by (i) sampling $(e_i)_{i=1}^{n+1}$, (ii) forming $(Y_i)_{i=1}^n$, (iii) sampling $\sigma \in \Pi_n$, and (iv) setting $X_i = Y_{\sigma(i)}$. 
    
    So, given a length $n$ continuous finite stream embedding $\alg^n = (\tau_i)_{i=1}^k$, we define a length $n$ permutation stream embedding $\alg^\sigma = (\tau_i')_{i=1}^k$ by setting
    \[\tau_i'(X_0,\sigma(1),\sigma(2)\dots, \sigma(n)) = \tau_i(X_0,Y_{\sigma(1)},Y_{\sigma(2)},\dots, Y_{\sigma(n)}).\]
    By construction, for any $i_1,\dots, i_m$, the sequence $(Y_{\sigma(i_j)})_{j = 1}^m$ has the same order type as $(\sigma(i_j))_{j=1}^m$, so the stopping times defined above give a well-defined permutation stream embedding algorithm.
    Taking $n=\beta^f(\pi)$ and choosing $\alg^n$ to witness the definition of $\beta^f(\pi)$, the coupled algorithm has the same success probability. Hence
    \[
        \pbeta(\pi)\le\beta^f(\pi).
    \]
\end{proof}

\begin{proof}[Proof of \cref{prop:Stopi}]
    In view of \cref{lem:couppiS}, it suffices to show that
    \[\beta^f(\pi) \le (1 + o_k(1))\beta(\pi).\]
    Consider the canonical algorithm $\alg^* = (\tau_i)_{i=1}^k$ achieving $\beta(\pi)$, and let $n = ( 1 + \eps)\beta(\pi)$. We may view $\alg^*$ as a length $n$ continuous finite stream embedding $\alg^n = (\tau_i')_{i=1}^k$ by setting $\tau_i' = n+1$ on the event $\tau_i > n$. \Cref{prop:variance-canonical} with Chebyshev's inequality implies
    \[\PP{\tau_k' = n + 1} = O(\eps^{-2}\beta(\pi)^{-1/2}).\]
    Since $\beta(\pi) \geq k^2/4$ (see \cref{thm:numerical-scaling}), we find that $\beta^f(\pi) \leq n$ if $\eps$ tends to zero slower than $k^{-1/2}$.
\end{proof}

\bibliographystyle{plain}
\bibliography{references}

\end{document}